\documentclass[onecollarge,dvipsnames]{article}                    
\usepackage{soul}
\usepackage[english]{babel}
\usepackage{graphicx}
\usepackage{comment}
\usepackage{amsmath,amssymb,amsbsy,latexsym}
\usepackage{color}
\usepackage{caption}
\usepackage{blindtext}
\usepackage{a4wide}
\usepackage{amsfonts, mathrsfs}
\usepackage{tikz}
\usepackage{verbatim}
\usepackage[normalem]{ulem}
\usepackage{url}
\usepackage{cancel}

\usepackage{subcaption}

\usepackage{tablefootnote}

\usepackage{enumitem}

\usepackage[ruled]{algorithm2e} 

\usepackage{cancel}
\usepackage{multirow}

\newtheorem{theorem}{Theorem}

\newtheorem{lemma}[theorem]{Lemma}
\newtheorem{example}[theorem]{Example}

\newtheorem{definition}[theorem]{Definition}

\newtheorem{remark}[theorem]{Remark}
\newenvironment{proof}[1][Proof]{\textbf{#1.} }{\ \rule{0.5em}{0.5em}}

\newcommand{\mat}{\mathbf}
\newcommand{\onevec}{\mat 1}

\newcommand{\E}{\mathbb E}

\renewcommand{\P}{\mathbb P}
\newcommand{\sgn}{{\rm sgn}}

\newcommand{\bG}{{\bf G}}

\newcommand{\bX}{{\bf X}}
\newcommand{\bT}{{\bf T}}

\newcommand{\hbX}{{\hat{\bf X}}}
\newcommand{\hX}{{\hat X}}

\newcommand{\argmin}{\mathop{\mathrm{argmin}}\limits}

\begin{document}

\title{Red Light Green Light Method for Solving Large Markov Chains}

\author{Konstantin Avrachenkov\thanks{Inria Sophia Antipolis, France},
Patrick Brown\thanks{Inria Sophia Antipolis, France} and
Nelly Litvak\thanks{University of Twente and Eindhoven University of Technology, The Netherlands}}

%
%


\maketitle

\begin{abstract}
Discrete-time discrete-state finite Markov chains are versatile mathematical models for a wide range of real-life stochastic processes. One of most common tasks in studies of Markov chains is computation of the stationary distribution. Without loss of generality, and drawing our motivation from applications to large networks, we interpret this problem as one of computing the stationary distribution of a random walk on a graph. We propose a new controlled, easily distributed algorithm for this task, briefly summarized as follows: at the beginning, each node receives a fixed amount of cash (positive or negative), and at each iteration, some nodes receive `green light' to distribute their wealth or debt proportionally to the transition probabilities of the Markov chain; the stationary probability of a node is computed as a ratio of the cash distributed by this node to the total cash distributed by all nodes together. Our method includes as special cases a wide range of known, very different, and previously disconnected methods including power iterations, Gauss-Southwell, and online distributed algorithms. We prove exponential convergence of our method, demonstrate its high efficiency, and derive scheduling strategies for the green-light, that achieve convergence rate faster than state-of-the-art algorithms.
\end{abstract}

\section{Introduction}

Markov chains are versatile models that provide a mathematical representation for a wide variety of real-life stochastic processes.
One of most common and important tasks in studies of Markov chains is the computation of a stationary distribution. Traditional applications include performance evaluation of queueing and computing systems, and analysis of chemical and biological systems. From the end of the 1990s, computation of Google PageRank~\cite{Brin1998anatomy} attracted vast attention in the literature and motivated the search for new numerical methods. The topic continued attracting much attention  due to many recent successful applications of PageRank
\cite{Gleich2015pagerank}
 well beyond its original use in web search.

A common challenge for modern applications of Markov chains is the extremely large number of states. For example, PageRank is a stationary distribution of a random walk on a large network, where the number of states (network nodes)  can be in the order of billions. These large Markov chains require computation methods with low complexity, small storage requirements, and  distributed implementation. Few existing approaches simultaneously satisfy these three criteria.

In this paper we propose a new general method for computing the stationary distribution of a large Markov chain. In the network setting, our algorithm can be pictorially described as follows.
At the beginning, each node carries a certain amount of cash that can be positive or negative (wealth or debt), and the total cash in the system equals zero. Then, at each step, chosen  nodes receive a green light to distribute their cash (wealth or debt) to other nodes proportionally to the transition probabilities of the Markov chain. It can be shown that the cash will tend to zero and that the total sum of routed cash converges to the stationary distribution by a multiplicative factor. This mechanism mimics  a children's game  {\it Statues}, also called {\it RLGL (Red Light, Green Light)},  where players, every so often, receive green light to move towards the goal while others stay fit; hence, our suggested name RLGL for this method.

Due to the presence of positive and negative cash, and the choice of which nodes receive a green light, the proposed method presents a  flexibility with interesting consequences. \ The RLGL algorithm incorporates many successful existing methods and extends the range of application and improves others: the classical widely applicable power iterations, the fast Gauss-Southwell algorithms for PageRank computation \cite{mcsherry2005uniform} requiring de facto a substochastic routing matrix, and the general on-line OPIC algorithm~\cite{Abiteboul2003OPIC} which in its original form  converges at a much slower than exponential rate. Moreover, we will demonstrate that RLGL achieves a higher convergence rate than known methods with simple green light scheduling strategies.


\subsection{Related literature}

Research on computation of a stationary distribution of Markov chains has a long history. Our goal here is to sketch the line-up of the results that are directly relevant for positioning of our work.  For a more comprehensive survey of this vast research area we refer the reader to several excellent books, for example,~\cite{stewart1991numerical,bini2005numerical}.

The stationary distribution of a Markov chain is a normalized solution of a linear system, or the principal eigenvector of the transition matrix. There are two general computational approaches for finding such solution: exact and iterative.

One of most successful exact methods is the Grassmann-Taksar-Heyman (GTH) method~\cite{GTH85}. This method is a variant of Gaussian elimination with no subtractions, and thus
no loss of significant digits caused by cancellation \cite{o1993entrywise}. Even though this method has remarkable performance for small and medium-sized Markov chains, it does not scale to large Markov chains due to the high computational complexity of the Gaussian elimination.

The iterative methods come in a much greater variety. The most straightforward one is the power iterations where the probability vector is iteratively multiplied by the transition matrix till convergence. The power iterations method is used in many applications thanks to its simplicity, suitability for sparse computations, and natural distributed implementations \cite{lubachevsky1986chaotic,BertsekasTsitsiklis1989}. Power iterations converge exponentially fast, however, the rate is limited by the modulus  of the second largest eigenvalue of the transition matrix. This results in a slow convergence when this number is close to unity. Since the stationary distribution of a Markov chain is merely a solution of a linear system, one can apply general purpose Krylov subspace family of methods, in particular, the Generalized Minimal Residual (GMRES) method, that is currently the state-of-the-art for solving linear systems \cite{Saad1986GMRES,Stewart1992GMRES_Markov}. We will numerically compare our RLGL method to both PI and GMRES.

Many specific iterative methods have been designed for solving Markov chains with particular structure. For instance, if a Markov chain consists of several weakly connected subchains, or clusters of states, then the aggregation-disaggregation approach is highly efficient, see e.g. the book~\cite{stewart1991numerical} and the many references therein. Another set of methods has been designed for queueing applications, where the transition matrix has distinct repeating patterns in its structure. Examples and many references can be found in the books~\cite{latouche1999introduction,bini2005numerical}. It is important to realize that high regularity in the structure of the Markov chain is necessary in order to benefit from these methods.

In the case of non-ergodic Markov chains with several closed classes of states and a set of transient states,
\cite{Berkhout2019jump} proposed a method for computing  the ergodic projection, using  the powers of transition matices.  Since  computing powers of matrices may not be feasible for very large sparse Markov chains, we expect that our RLGL method can serve as a subroutine in the task of computing ergodic projection.

In 1998, the introduction of PageRank by the Google founders~\cite{Brin1998anatomy} gave an enormous impetus to the development of new computational methods for solving very large Markov chains. PageRank is a Markov chain on the web graph, with restart. It is noteworthy that many of the developed methods rely on the restart property that yields a guaranteed convergence rate for iterative methods. They are thus not applicable to general Markov chains. Below we will focus on methods directly related to our algorithm. For a broader overview on the PageRank computations we refer the reader to the surveys~\cite{ishii2018distributed,park2019survey,Langville}.

Out of the many algorithms for PageRank computations, the family of Gauss-Southwell methods~\cite{mcsherry2005uniform,Andersen06,berkhin2006bookmark,Hong2012,hong2015d,nassar2015strong,suzuki2018distributed} and other versions of coordinate descent (see e.g., \cite{DaiFreris2017}), clearly stand out due to their rapid convergence, often much faster than  power iterations. Such methods have been originally proposed in 1940-s for solving linear systems by greedy elimination of the absolute value of the residual~\cite{SouthwellRelaxationEng,SouthwellRelaxationPhys}.
In Section~\ref{ssec:comparison} we will demonstrate that  the Gauss-Southwell methods for PageRank computation are a particular case of our RLGL algorithm.
In fact, our work enables the application of Gauss-Southwell type methods for solving general Markov chains, and provides the  proof of its convergence.


A recent application of PageRank is the Personalized PageRank (PPR) with restart occurring from the same node. In this case the random walk typically stays close to the restart node, which makes this model useful for local graph clustering~\cite{Andersen06,Spielman2013local}, similarity measures \cite{Avrachenkov2019similarities} and semi-supervised learning~\cite{Avrachenkov2012semisupervised}.
Several algorithms~\cite{Andersen06,Spielman2013local,nassar2015strong,vial2019structural} take advantage of localization of PPR to compute its sparse approximation. Moreover, the recently developed FAST-PPR method~\cite{Lofgren2014fastPPR} achieves even faster convergence by combining the Gauss-Southwell-type method from \cite{Andersen2007local} with random walks.

When solving large Markov chains, it is desirable to have algorithms that allow for distributed implementation. Many distributed algorithms exist for solving large linear systems~\cite{BertsekasTsitsiklis1989,Baudet1978Asynchronous,avron2015revisiting,liu2017asynchronous}. Although Markov chains, and PageRank in particular, are special cases of large linear systems, it pays off to leverage on their very specific properties. An overview of distributed methods specifically designed for PageRank is given recently in~\cite{ishii2018distributed}. In addition to these methods, there are also online algorithms for PageRank computations which allow for distributed implementation. This includes Monte-Carlo algorithms~\cite{Fogaras2005towards,Avrachenkov2007MC}, based on simulation of random walks, and the OPIC algorithm~\cite{Abiteboul2003OPIC,Litvak2008Performance,Litvak2012AAP}, based on redistribution of positive cash. Our proposed RLGL algorithm is inspired by OPIC. We will explain this algorithm and compare it to our RLGL method in Section~\ref{ssec:comparison}. There we also discuss in detail the connection between RLGL and the closely related Gauss-Southwell methods.

\subsection{Contributions and the structure of the paper}

In the next Section~\ref{sec:algorithm} we present our RLGL algorithm. Section~\ref{ssec:def_RLGL} provides the formal description of the algorithm. Then, Section~\ref{ssec:math_RLGL} explains the mathematical rationale behind this method  and its novelty. Section~\ref{ssec:comparison} compares the RLGL method to the power iterations, and the two most relevant approaches: Gauss-Southwell and OPIC. We show that RLGL generalizes the above mentioned three methods.

In Section~\ref{sec:exp_convergence} we establish general sufficient conditions for the exponential convergence of the RLGL algorithm.
Our results generalize the Gauss-Southwell method beyond PageRank to general Markov chains, while establishing its exponential convergence. To the best of our knowledge, in the literature, the proofs of  exponential convergence of the Gauss-Southwell-type methods always use the restarting property of PageRank. So Gauss-Southwell methods solve for:
$$\pi^*=\pi^* P + b,$$
where $P$ is a substochastic matrix and $b$ is a vector. Since our aim is to deal with general Markov chains, our proofs are built on completely different arguments. We show that, as a consequence of positive and negative cash in RLGL, the error term can be expressed as the total variation distance between two Markov chains. Depending on the scheduling rule of green light, exponential convergence can be proved directly or using a coupling argument. When the schedule is randomized, the proof of exponential convergence also relies on large deviation bounds.


In Section~\ref{sec:two_blocks} we analyze a very simple mean-field stochastic block model (SBM) with two blocks to demonstrate that the RLGL algorithm can converge significantly faster than power iterations. This simple example already provides valuable insights about effective scheduling of the green light. Specifically, in the case of two blocks, an intuitive rule to give the green light  to the nodes with higher absolute value of cash appears to yield fast convergence.

However, the optimal scheduling of green light in general Markov chains leads to a dynamic program with exponential number of actions in the system size. Surprisingly, even when going from two to three clusters in the mean-field SBM, the optimal scheduling of green light exhibits a very subtle dependency on the cash distribution between the blocks. In Section~\ref{sec:three_blocks} we  set up and solve the dynamic program for this model numerically using the value iteration method.
These examples demonstrate the speed of convergence may be greatly improved by optimizing the green light scheduling, opening an avenue for further research.

Section~\ref{sec:num} contains numerical results.  We demonstrate the fast convergence of the RLGL method and experimentally compare different scheduling policies for the green light. Then,
we compare RLGL with the other state of the art methods and demonstrate that RLGL has superior performance.
In Section~\ref{sec:num} we also indicate distributable versions of the RLGL method.

Section~\ref{sec:conc} concludes the paper with a number of future research directions.

\section{The RLGL algorithm}
\label{sec:algorithm}

\subsection{Formal description of the RLGL algorithm}
\label{ssec:def_RLGL}

Consider an ergodic Markov chain with a finite state space $[N]=\{1,2,\ldots,N\}$ and transition probability matrix $P=(p_{ij})$, where $p_{ij}$ is the probability that the next state is $j$ given that the current state is $i$.
Let  $\pi^* = (\pi^*_1,\pi^*_2,\ldots,\pi^*_N)$ be the stationary distribution of this Markov chain, so $\pi^*$ is the solution of:
$$\pi^*=\pi^* P,$$
with $\sum_{i=1}^N\pi^*_i =1$.  The {\it goal} of the RLGL algorithm is to quickly compute $\pi^*$. {One desired aspect} is to do so in a distributed way, that is, computations are executed locally by a node, or a group of nodes.

We number the steps of the algorithm by $t=0,1,\ldots$.   For each $t\ge 0$ and $i\in [N]$, let $C_{t,i}\in {\mathbb{R}}$ be the amount of cash at node $i$ at the {\it beginning} of step $t$.
Denote $C_t = (C_{t,1}, C_{t,2},\ldots,C_{t,N})$.

 At step $t\ge 0$, a set of nodes $G_t\subseteq [N]$ receive a `green light', and  move cash to their neighbors proportionally to the transition probabilities. Let $M_{t,i}$ be the amount of cash moved by node $i$ at step $t$, and denote $M_t = (M_{t,1}, M_{t,2},\ldots, M_{t,N})$. Then the cash of each node is updated as follows:
\begin{align}
\label{eq:grl_update_C}
C_{t+1,i}&=C_{t,i}-M_{t,i}+ \sum_{j=1}^Np_{ji}M_{t,j}, \quad t\ge 0,\; i\in[N].
\end{align}
Step~0 is special. We set $G_0=[N]$, and set $M_0$ to be equal to some probability distribution on $[N]$, which usually will be the uniform distribution: $M_0=(1/N)\, {\bf 1}^T$, where ${\bf 1}$ is the column-vector of ones. We set $C_0={\bf 0}^T$, where ${\bf 0}$ is a column-vector of zeros.
 After moving the amount $1/N$, each  node's cash becomes negative, $-1/N$, and this is (possibly partially) compensated by the positive cash received from other nodes, according to (\ref{eq:grl_update_C}). As a consequence, some nodes will have positive cash and others negative cash. Since $P$ is a stochastic matrix the total cash is equal to zero. At subsequent steps $t\ge 1$, nodes in $G_t$ move all their cash, positive or negative, while the nodes outside $G_t$ do not move cash. Formally, $M_{t,i}=C_{t,i}$ if $i\in G_t$, and $M_{t,i} = 0$ otherwise. Note that, since there is no in- or out-flow of cash, the total amount of cash in the system remains zero at all times.

The {\it history} $H_{t,i}$ of node $i$ at the beginning of step $t$ is defined as the total amount of cash moved by $i$ before step~$t$:
\[H_{0,i}=0, \quad H_{t,i} = \sum_{k=0}^{t-1}M_{k,i},\quad t\ge 0, \quad i\in [N],\]
and $H_t=(H_{t,1}, H_{t,2},\ldots,H_{t,N})$. Note that from (\ref{eq:grl_update_C}) we have:
$$H_t = -C_t+ H_tP.$$
At step $t\ge 1$, the algorithm provides the estimation $\hat \pi_{t,i}$ of $\pi^*_i$ as a ratio between $H_{t,i}$ with and the total history:
\begin{equation}
\label{eq:pi_hat}
\hat\pi_{t,i} = \frac{H_{t,i}}{H_{t} {\bf 1}},  \quad i\in [N], \; t\ge 1,
\end{equation}
and we denote $\hat \pi_t = (\hat \pi_{t,1}, \hat \pi_{t,2},\ldots,\hat \pi_{t,N})$.

Note that $\hat \pi_{t,i}$ can be negative or greater than one, so, in general, $\hat \pi_{t,i}$ is {\it not} a probability measure, even though $\hat \pi_{t}{\bf 1} =1$ provided that $H_{t} {\bf 1}\ne 0$. Furthermore, $H_{t} {\bf 1}$ can be zero, in which case we restart the algorithm and/or change the choices of $G_t$. In our numerical experiments, on realistic examples, this did not happen, but we present a counterexample where it may occur depending on the green light scheduling. We will discuss this in more detail in Section~\ref{sec:exp_convergence}.

The algorithm stops when convergence occurs:  $||\hat \pi_{t}-\hat \pi_{t}P||_1<\varepsilon$ for some desired $\varepsilon>0$.  In the next Section~\ref{ssec:math_RLGL} we will explain why convergence occurs and how this leads to the  idea behind the RLGL algorithm.

We close this section by a formal summary of the RLGL algorithm. Let $I$ denote the identity matrix, and let $I(G_t)$ be the identity matrix on $G_t$ and zero elsewhere, so $I(G_t)_{ii}=1$ when $i\in G_t$, and $I(G_t)_{ij}=0$ otherwise. Then the RLGL algorithm computes $M_t$, $C_t$ and $H_t$ recursively as follows:
\begin{align}
\label{eq:init}
&C_0={\bf 0}, H_0={\bf 0}, M_0=(1/N) {\bf 1},\\
\label{eq:recursion_M}
&M_{t}=C_{t}I(G_t), \quad t\ge 1,\\
\label{eq:recursion_H}
&H_{t+1}=H_{t} + M_{t}, \quad t\ge 0,\\
\label{eq:recursion_C}
&C_{t+1}=C_{t}-M_{t} + M_tP = C_t - M_t(I-P), \quad t\ge 0.
\end{align}
Algorithm~\ref{alg:RLGL} provides a pseudocode of the RLGL algorithm.

\begin{algorithm}[htb]
\caption{\textsc{Red Light Green Light Algorithm (RLGL)}}
\label{alg:RLGL}
{Input:} $N$, $P$, $\varepsilon$, rule for choosing $G_t$

{Output:} $\hat \pi_{t}$

\begin{enumerate}[noitemsep,nolistsep]
\item[(1)] $t=0$, $C_0={\bf 0}$, $H_0={\bf 0}$, $M_0=\frac{1}{N} \, {\bf 1}$;
\item[(2)] $H_{t+1}=H_{t} + M_{t}$;
\item[(3)]  $C_{t+1}=C_{t}-M_{t}(I-P) $;
\item[(4)]  $t \leftarrow t+1$;
\item[(5)] $\hat \pi_{t} = \frac{1}{H_t {\bf 1}}\, H_t$;
\item[(6)] $M_{t}=C_{t}I(G_t)$;
\item[(7)] Iterate (2) -- (6) until $||\hat \pi_{t}-\hat \pi_{t}P||_1 =||\frac{1}{H_t {\bf 1}}\, C_t||_1 < \varepsilon$ for some $t\ge 2$.
\end{enumerate}
\end{algorithm}

\subsection{Rationale behind the RLGL algorithm}
\label{ssec:math_RLGL}

The idea behind the RLGL algorithm emerges from the mathematical description of its dynamics, in the same lines as in~\cite{Abiteboul2003OPIC,Litvak2008Performance} for the OPIC algorithm. To begin with, without imposing any assumptions on $C_0$, one can write the equation for the cash balance in the system:
\begin{equation}
\label{eq:recursion_H_res}
H_{t}+C_{t} = M_{0} + H_{t}P + (C_{0} - M_{0}) =  H_{t}P + C_{0}, \quad t\ge 1,
\end{equation}
which is explained as follows. The coordinate $i$ of the vector on the left-hand side is the total amount of cash that $i$ has ever possessed until the beginning of step $t$. This consists of the three terms on the right-hand side: i) the amount moved by $i$ at step $t=0$, $M_{0,i}$; ii) all cash received by $i$ from other nodes, $\sum_{j=1}^N H_{t,j} p_{ji}$, and iii) the cash left after step~0, that is, $C_{0,i}-M_{0,i}$. Canceling $M_{0}$ on the right-hand side of \eqref{eq:recursion_H_res}, dividing both sides by $H_{t}{\bf 1}$,  and substituting $\hat \pi_{t} = \frac{H_t}{H_{t}{\bf 1}}$,  we obtain
\begin{align}
\label{eq:pi*}
&\hat \pi_{t} =\frac{1}{H_{t}{\bf 1}}(C_0-C_t) + \hat \pi_{t} P, \quad \hat \pi_{t}{\bf 1}  =1, \quad t\ge 1.
\end{align}
It is easy to check that this system of linear equations has a unique solution $\hat\pi_t$, which satisfies
\begin{equation}
\label{eq:convergence}
\hat \pi_t - \pi^* = \frac{1}{H_{t}{\bf 1}}\left(C_0 - C_t\right)\sum_{k=0}^\infty(P^k - {\bf 1} \pi^*).
\end{equation}
The idea of the RLGL algorithm is to improve convergence by making $||C_0 - C_t||_1$ converge to zero quickly. For that, the cash $C_t$ must be driven as close as possible to its initial value $C_0$. Hence, the idea is that node $i\in G_t$ moves the difference $C_{t,i}-C_{0,i}$ of cash, positive or negative. As negative difference is routed to positive difference, and reciprocally, $||C_0 - C_t||_1$ decreases. This differentiates RLGL from Gauss-Southwell and OPIC methods which only move positive cash $C_t$. Thus Gauss-Southwell must rely on substochasticity for convergence while OPIC obtains the convergence of the error term in (\ref{eq:pi*}) by having $H_t{\bf 1}$ tend to infinity, resulting in slow convergence. In the following we adopt the equivalent formulation where $C_0={\bf 0}$ and $M_t$ is an arbitrary probability distribution.

We emphasize that it is crucial that the cash is moved proportionally to transition probabilities at all steps $t\ge 0$. Otherwise, the term $H_tP$ on the right-hand side of \eqref{eq:recursion_H_res} will not be valid, and the algorithm will converge to a wrong limit.

To get an idea about convergence of $||C_t||_1$, denote by
\[P(G_t) = I-I(G_t)(I-P),\quad t\ge 0,\]
the matrix that describes the cash movement: its rows $i\in G_t$ are the same as in $P$ and rows $i\notin G_t$ are as in the identity matrix. Next,  use \eqref{eq:init} and \eqref{eq:recursion_M}, and iterate \eqref{eq:recursion_C} to obtain
\begin{align}
\nonumber
C_1&=-M_0(I-P),\\
\label{eq:cn}
C_{t} & = C_{t-1}[I-I(G_{t-1})(I-P)]=C_{t-1}P(G_{t-1})=-M_0(I-P)\prod_{k=1}^{t-1}P(G_k),\quad t\ge 2.
\end{align}
We see that $C_{t}$ is expressed through a product of stochastic matrices, making exponential convergence plausible. Indeed, in Section~\ref{sec:exp_convergence} we will prove the exponential convergence of $||C_t||_1$ under very general conditions on $P$ and $\bG=(G_0,G_1,\ldots)$.

\begin{remark}
We note that~\cite{bajovic2012products} obtains the exponential convergence for the product of symmetric stochastic matrices, however, their argument does not extend to general Markov chains. The
general results on inhomogeneous Markov chains from \cite{seneta2006book} also do not apply,
at least directly, as they require conditions on regularity of $P(G_k)$, which we do not have
unless $G_t=[N]$. To the best of our knowledge, exponential convergence of the product of stochastic matrices as in (\ref{eq:cn}) has not been proven before.
\end{remark}

Importantly, in the RLGL setting, one must propose a sequence of sets $\bG$ that achieve efficient reduction of $||C_0 - C_t||_1$. Moreover, $G_t$ can be chosen in an unfortunate way so that no convergence occurs, for example, when $G_t=\{1\}$ for all $t=1,2,\ldots$, so that only node 1 moves its cash. We will see a more intricate example in Section~\ref{sec:three_blocks}.

\subsection{Comparison to other methods}
\label{ssec:comparison}

\subsubsection{Power Iterations}

In a special case $G_t=[N]$, we have $P(G_t)=P$ for all $t\ge 0$. Then by iterating \eqref{eq:recursion_H}, from \eqref{eq:recursion_M} and \eqref{eq:cn} we obtain
\begin{align*}
H_{t} &= \sum_{k=0}^{t-1} M_k = M_0+\sum_{k=1}^{t-1}C_k = M_0 - M_0 (I-P) \sum_{k=1}^{t-1} P^{k-1}= M_0P^{t-1}, \quad t\ge 1.
\end{align*}
By choosing $M_0$ to be a probability distribution over $[N]$, we have that $\hat \pi_{t} = M_0P^{t-1}$, therefore power iterations are a special case of the RLGL algorithm.

\subsubsection{Gauss-Southwell method}

We next show that RLGL is a generalization of the Gauss-Southwell algorithm for PageRank (GSo-PR). We will discuss this in detail because GSo-PR is a state-of-the-art approach for PageRank computation, and, to the best of our knowledge, its connection to probabilistic on-line methods, such as OPIC, was not established before.

PageRank $\pi^*$ is a stationary probability of a Markov chain that, at each time step, with probability $c\in (0,1)$ follows a transition matrix $P$, and with probability $(1-c)$ restarts from a random node sampled from probability distribution ${\bf s} = (s_1,s_2,\ldots,s_N)$. Then $\pi^*$ solves the following linear system:
\begin{equation}
\label{eq:pagerank}
\pi^* = c \pi^* P + (1-c) {\bf s}.
\end{equation}
Let us now explain the  GSo-PR method. We will intentionally use the same notations as for the RLGL algorithm so that the connection between the two algorithms becomes transparent. Let $H_t$ be the estimate of $\pi^*$ at time $t$. The idea of the GSo-PR method is to iteratively reduce the {\it residual}, $C_t$, defined as \begin{equation}
\label{eq:residual}
C_t = c H_t P + (1-c) {\bf s} - H_t, \quad t\ge 0.
\end{equation}
Clearly, when $C_t={\bf 0}$ we have found the exact solution of \eqref{eq:pagerank}.
In the literature~\cite{mcsherry2005uniform,nassar2015strong,Hong2012,hong2015d,suzuki2018distributed}, GSo-PR algorithm is initialized with $C_1=(1-c){\bf s}$ and $H_1 = {\bf 0}$. Then at each iteration, some probability mass is transferred from $C_t$ to $H_{t+1}$. Formally, let ${\bf e}_k$ be the column-vector with the $k$-th coordinate equal to 1 and all other coordinates equal to zero, and set $M_t=C_{t,k}{\bf e}_k^T$. Usually the greedy coordinate descent is used, so $k=\arg\max_{i\in [N]} C_{t,i}$. Then the estimate $H_t$ is updated, namely, its $k$-th coordinate increases by $C_{t,k}$:
\begin{equation}
\label{eq:GSo_H}
H_{t+1}=H_t+M_t, \quad t\ge 1,
\end{equation}
and $C_{t+1}$ is updated by substituting $H_{t+1}$ in the right-hand side of \eqref{eq:residual}:
\begin{equation}
\label{eq:GSo_C}
C_{t+1}= C_t - M_t(I-c P), \quad t\ge 1.
\end{equation}
A pseudocode of the Gauss-Southwell method is given in Algorithm~\ref{alg:GSo}.

\begin{algorithm}[htb]
\caption{\textsc{Gauss-Southwell algorithm for PageRank computation}}
\label{alg:GSo}
{Input:} $N$, $P$, $\varepsilon$, ${\bf s}$, $c$

{Output:} $\hat \pi_{t}$

\begin{enumerate}[noitemsep,nolistsep]
\item[(1)] $t=1$, $C_1= (1-c){\bf s}$, $H_1={\bf 0}$;
\item[(2)] $k = \arg\max_{i\in [N]} C_{t,i}$, $M_t = C_{t,k}{\bf e}_k$
\item[(3)] $H_{t+1}=H_{t} + C_{t,k}{\bf e}_k$;
\item[(4)] $C_{t+1} = C_t  - M_t(I-cP)$;
\item[(5)]  $t \leftarrow t+1$;
\item[(6)] $\hat \pi_t = H_t$;
\item[(7)] Iterate (2) -- (6) until  $||C_t||_1< \varepsilon$ for some $t\ge 2$.
\end{enumerate}
\end{algorithm}

Notice that the updating rules \eqref{eq:GSo_H}, \eqref{eq:GSo_C} for the GSo-PR method are identical to the updating rules \eqref{eq:recursion_H},\eqref{eq:recursion_C} of the RLGL algorithm.
Furthermore, Algorithm~\ref{alg:GSo} converges to a probability distribution~\cite{mcsherry2005uniform}, therefore estimator $H_t$ in Algorithm~\ref{alg:GSo} is equivalent to the normalized estimator $\hat \pi_t = \frac{H_t}{H_t{\bf 1}}$ in Algorithm~\ref{alg:RLGL}.  However, while the routing matrix $P$ in \eqref{eq:recursion_C} is stochastic, the routing matrix $cP$ in \eqref{eq:GSo_C} is substochastic and GSo-PR cannot be applied to arbitrary Markov chains in contrary to RLGL.
 Now, let us show that GSo-PR is in fact a special case of RLGL. For that, we introduce an auxiliary node~0, and define a transition matrix $P^{(0)}$, which extends $P$ to an $(N+1)\times (N+1)$ matrix as follows:
\[P^{(0)} = \left(\begin{array}{c|c}c& (1-c){\bf s}\\ \hline (1-c){\bf 1}& c P\end{array}\right).\]
Next, in the RLGL Algorithm~\ref{alg:RLGL} we set $M^{(0)}_0 = {\bf e}_0^T$, $C^{(0)}_0={\bf 0}^T$,
and apply the RLGL algorithm to $P^{(0)}$ (we will use the upper index $(0)$ in all corresponding notations to indicate the modified system). Then we get
\begin{align*}
C^{(0)}_1 &= M^{(0)}_0(P^{(0)}-I) = (- (1-c), (1-c){\bf s}),\\
H^{(0)}_1 &= M^{(0)}_0=(1, {\bf 0}).
\end{align*}
For the nodes $1,2,\ldots,N$, this is exactly the initialization of Algorithm~\ref{alg:GSo}.
The choice of next cash movement in Algorithm~\ref{alg:GSo} is a particular green light schedule
in Algorithm~\ref{alg:RLGL} with $G_t=\{k\}$, $k \neq 0$ and $k$ chosen according to the
optimization procedure specified in line (2) of Algorithm~\ref{alg:GSo}.
Furthermore, the update in line (4) of Algorithm~\ref{alg:GSo} is identical to the update in line (3)
of Algorithm~\ref{alg:RLGL} since the $N\times N$ right lower corner block of matrix $P^{(0)}$ is equal to $cP$.
Thus, by choosing this special setting for the RLGL method, we make it equivalent to the GSo-PR.

The crucial difference between GSo-PR and RLGL is that RLGL allows positive and negative cash at all nodes. In contrast, the GSo-PR algorithm has the entire negative cash in the auxiliary node~0, while the authentic nodes~$1,2,\ldots,N$ can have only positive cash, and only positive cash can be moved. Therefore, the RLGL algorithm offers much greater flexibility and substantial generalization compared to the GSo-PR method. As we shall see in Section~\ref{sec:num}, when applying RLGL to PageRank computation, one
can choose a more effective green light scheduling strategy in comparison with that of GSo-PR,  that can significantly improve the rate of convergence.

Moreover, our results make it possible to apply the family of Gauss-Southwell-type methods to the solution of general Markov chains, which greatly expands its applicability. Indeed, by construction, $C_t-C_0$ in \eqref{eq:recursion_H_res} is merely a residual, which we iteratively reduce to zero. To the best of our knowledge,~\cite{mcsherry2005uniform} is the only work that attempts to apply a Gauss-Southwell-type method to general Markov chains, but it builds on GSo-PR by splitting $P=A+B$, where $B$ is a rank-1 matrix, which in fact plays the role of the restart and defines the speed of convergence. This is much more restrictive than our proposed method, of which performance does not depend on any special representation of $P$.
We emphasize that to the best of our knowledge there is no proof of convergence of the Gauss-Southwell method neither for general linear systems, nor for solving Markov chains. Our approach uses the normalization with the history process $H_t$, together with interpretation of the process $C_t$ through the coupling of two Markov chains. Together, these techniques were not used before, and  enable the formal proof of convergence.

\subsubsection{OPIC algorithm}

The OPIC algorithm~\cite{Abiteboul2003OPIC,Litvak2008Performance} was our initial starting point for developing the RLGL method. In contrast to the RLGL, OPIC fixes positive $C_0$ such that $||C_0||_1=1$. At step $t\ge 1$, in OPIC, exactly as in RLGL, nodes in $G_t$ move all their cash to their neighbors. Thus, in OPIC the cash at every node $i$ remains non-negative so $||C_0-C_t||_1$ remains of the order $O(1)$ as $t\to\infty$. In this case, convergence is achieved because the total history $H_{t}{\bf 1}$ in the denominator of the error-term in the right-hand side of (\ref{eq:convergence}) grows to infinity. Since the total amount of cash in the system is bounded, the total history grows to infinity at most linearly in $t$, which results in the speed of convergence $O(1/t)$.

\section{Exponential convergence of the RLGL algorithm}
\label{sec:exp_convergence}


In this section we investigate the exponential convergence of the RLGL algorithm. This is mainly determined by the exponential convergence of $||C_t||_1$ to zero. In Section~\ref{ssec:general_conditions} we establish general sufficient conditions for the exponential convergence of $||C_t||_1$ and discuss how this result can be applied using the total variation distance between two Markov chains. In Sections~\ref{ssec:dobrushin} -- \ref{ssec:geometric} we use this approach to establish exponential convergence and evaluate its rate in three natural special cases.

\subsection{General sufficient conditions for the exponential convergence}
\label{ssec:general_conditions}

Recall that $||C_t||_1$ is a possibly random function of sequence $\bG$ of sets that receive green light. The next lemma formalizes the fact that exponential convergence of the RLGL algorithm is determined by the exponential convergence of $\E(||C_t||_1)$.

\begin{lemma}\label{th:general} Assume that the RLGL algorithm is applied to the transition matrix $P$ of an ergodic Markov chain on state space $[N]$. If
\begin{equation}
\label{eq:condition_nu_P}
\E\left(||C_t||_1\right) \le a\rho^t, \quad t\ge 0,
\end{equation}
then for any function $f(t)> 0$ such that $\lim_{t\to\infty} f(t) =\infty$, it holds that
\begin{equation}
\label{eq:exp_conv_for_P}
\P\left(||C_t||_1\le f(t) \rho^t\right) = 1-o(1), \quad  \mbox{as $t\to\infty$}.
\end{equation}
If, in addition, $\inf_{t>0} |H_t{\bf 1}| >0$,
then \[||\hat \pi_t-\pi^*||_1= O_{\P}(f(t)\rho^t), \quad  \mbox {as $t\to\infty$},\]
where $O_{\P}(\cdot)$ means that the big-O relation holds in probability.
\end{lemma}

\begin{proof}
From Markov's inequality and \eqref{eq:condition_nu_P} we have:
\begin{align*}
\P(||C_t||_1> f(t) \rho^t)&\le \frac{\E(||C_t||_1)}{f(t) \rho^t} \le \frac{a}{f(t)}\to 0,
\quad \mbox{ as $t\to\infty$,}\end{align*}
which proves \eqref{eq:exp_conv_for_P}. Furthermore,
when $\inf_{t>0} |H_t {\bf 1}| = h>0$ assuming $C_0={\bf 0}^T$, by \eqref{eq:convergence} we have
\[
||\hat \pi_t - \pi^*||_1 \le \frac{||C_t||_1}{h}\left|\left|\sum_{k=0}^\infty(P^k - {\bf 1} \pi^*)\right|\right|_{1\to 1},
\]
where the subindex  $1 \to 1$ means that we are using the induced operator norm for an operator acting from one space with 1-norm
to another space with 1-norm.
The term $\sum_{k=0}^\infty(P^k - {\bf 1} \pi^*)$ is bounded, since the series is convergent
for an ergodic Markov chain (see Appendix~A.5 in \cite{Puterman}).
Thus, the RLGL algorithm converges exponentially at the same rate as $||C_t||_1$.
\end{proof}

\begin{remark}
\label{remark:det}
If $\bG$ is deterministic, then $||C_t||_1$ and $||\hat \pi_t-\pi^*||_1$ are deterministic as well, so we have $||C_t||_1=\E(||C_t||_1)\le a\rho^t$ and  $||\hat \pi_t-\pi^*||_1=O(\rho^t)$.
\end{remark}

The condition  $\inf_{t>0} |H_t{\bf 1}|>0$ is important because  $\hat \pi_t = \frac{H_t}{H_t{\bf 1}}$, so we need to separately address the case when $H_t{\bf 1} = 0$ for some $t>0$, or $\lim_{t\to\infty}H_t{\bf 1} = 0$.

In the particular case $H_t = {\bf 0}^T$ it follows from \eqref{eq:recursion_H_res} that $C_t-C_0={\bf 0}^T$, so the algorithm returns the trivial solution $H_t= {\bf 0}^T$ of \eqref{eq:recursion_H_res}. This can occur even when $P$ is aperiodic and irreducible, as in Example~\ref{ex:counterexample} below.

\begin{example}
\label{ex:counterexample}

Consider a Markov chain with four states $\{1,2,3,4\}$ and the following transition probability matrix:

\begin{equation}
\label{eq:counterexampleP}
P = \left(\begin{array}{rrrr} 0&0.5&0.5&0\\0&0&1&0\\0&0&0&1\\1&0&0&0\end{array}\right).
\end{equation}
It is easy to see that $P$ is aperiodic and irreducible, $\pi^*= (2/7, 1/7, 2/7, 2/7)$. Take $M_0 = (1,0,0,0)$. Then after the initial step of the algorithm, we have $C_1 = (-1, 0.5,0.5,0)$, $H_1=(1,0,0,0)$. Now, if $G_1=\{1\}$, then $H_2 = C_2 = (0,0,0,0)$.

\end{example}

Notice that if $||C_t||_1$ converges exponentially fast to zero, then $H_t{\bf 1}$  converges to some limit as $t\to\infty$. Since the cash can be positive or negative, it may happen that $H_t{\bf 1}$ is zero for some $t$ or its limit is zero as $t\to\infty$. Specifying precise conditions when it happens is beyond the scope of this paper. It is clear, however, that a very delicate balance is needed to make $H_t{\bf 1}$ to be equal  to or to converge to any specific value, including zero. In practice, if we observe $H_t{\bf 1}=0$, then we simply restart the algorithm if $\bG$ is stochastic, and modify $\bG$ slightly if $\bG$ is deterministic. In the numerous experiments performed for this study (see Section~\ref{sec:num}), we have never encountered the situation when $H_t{\bf 1}$ is zero or converges to zero.

For highly regular $P$ and $\bG$, \eqref{eq:condition_nu_P} might be easy to establish directly. We will see such example in Section~\ref{sec:two_blocks}, where we study the mean-field stochastic block model with two blocks. In more general cases, one must establish some kind of contraction of $\E(||C_t||_1)$. To this end, we will next show that $||C_t||_1$ equals twice the total variation distance between two specific Markov chains. This result is very useful because it gives access to an array of methods developed for proving  exponential convergence  of Markov chains, such as coupling and minorization conditions~\cite{rosenthal1995convergence}. We will use this for proving exponential convergence in Subsections~\ref{ssec:dobrushin}--\ref{ssec:geometric}. The next definition introduces the two Markov chains and their notations, which we will us throughout the paper.

\begin{definition}
\label{def:XX'}
Let  $\bX=(X_0, X_1,\ldots)$  be an inhomogeneous Markov chain with transition probability matrix $P(G_t)$ from $X_t$ to $X_{t+1}$. Define another inhomogeneous Markov chain $\bX'=(X_0', X_1',\ldots)$ as follows.  At $t=0$, $\bX'$  stays unchanged, $X'_1 = X'_0$. At $t\ge 1$, $\bX'$ obeys the same rules as $\bX$, so the transition from $X'_t$ to $X'_{t+1}$ occurs according to $P(G_t)$.
\end{definition}
In fact, each of the two Markov chains $\bX$ and $\bX'$ describes the movement of a `random particle' of cash. In both cases, when $t>0$, if a particle is in $G_t$, it moves according to the transition matrix $P$. The difference between $\bX$ and $\bX'$ is only in the first step,  at $t=0$. A particle guided by $\bX$ moves according to $P(G_0)=P$, while a particle guided by $\bX'$ does not move.

We denote by $\pi_t$ and $\pi_t'$ the probability distributions of, respectively, $X_t$ and $X_t '$ at time $t\ge 0$, and let
\[d_{TV}(\pi_t,\pi'_t)=\frac{1}{2}||\pi_t-\pi'_t||_1\]
be the total variation distance between $X_t$ and $X_t'$. Note that $\pi_t,\pi_t'$ and $d_{TV}(\pi_t,\pi'_t)$ depend on $\bG$, and can be random if $\bG$ is random.

Now notice that on the right-hand side of (\ref{eq:cn}), $M_0$ and $M_0P$ are probability vectors, and the matrix product contains the $(t-1)$-step transition probabilities of $\bX$ and $\bX '$. Hence, by setting  $\pi_0=\pi_0'=M_0$, we get $\pi_1= PM_0$ and $\pi_1'=M_0$, so we can express $C_t$ as follows:
\begin{align}
\label{eq:Cpi-pi'}
&C_t=\pi_t-\pi'_t,\\
\label{eq:C_tv0}
&||C_t||_1 = ||(\pi_1-\pi_1')\prod_{k=1}^{t-1}P(G_k)||_1 = 2d_{TV}(\pi_t,\pi'_t) = 2\E(d_{TV}(X_t, X_t')| \bG).
\end{align}

\begin{remark}  Notice that we need to set $\pi'_1=\pi_0$ and $\pi_1=\pi_0P(G_0)$ solely for the purpose of obtaining the term $H_tP$ on the right-hand side of \eqref{eq:recursion_H_res}, so that $\hat \pi_t$ converges to the correct limit $\pi^*$. We emphasize that the interpretation of $||C_t||_1$ through the total variation distance holds for any $\pi_1$ and $\pi_1'$.
\end{remark}

\begin{remark} {
Interestingly, we can immediately conclude that $d_{TV}(\pi_t,\pi'_{t})$ is non-increasing in $t$ because  $||C_t||_1$ is non-increasing in $t$ by the construction of the RLGL algorithm, as discussed in the proof of Theorem~\ref{th:T}. When $G_t=[N]$ for all $t\ge 1$, this gives an elegant, and, to the best of our knowledge, new proof that for two Markov chains $X_t$ and $X_t'$ with the same transition matrix $P$ and arbitrary initial distribution, $d_{TV}(X_{t},X'_{t})$ is  a non-increasing function of $t$.}
\end{remark}

The next Theorem~\ref{th:T} provides a useful upper bound for $\E(||C_t||_1)$ in terms of the number of contractions by some factor $\delta \in (0,1)$.

\begin{theorem}\label{th:T} Let $P$ be a transition matrix of an ergodic Markov chain on state space $[N]$, and let $\bX$ and $\bX'$ be as in Definition~\ref{def:XX'}. Assume that there exist $\delta\in (0,1)$, $a>0$, and an increasing, possibly random, sequence of time instants $(T_0=1, T_1, T_2,\ldots) ={\bf T} := {\bf T}(\delta)$, such that
\begin{equation}
\label{eq:alpha}
\E(d_{TV}(X_{T_n},X'_{T_n})|\bT) \le a\, \delta^n, \quad n\ge 1.
\end{equation}
Let $\nu(t)$ be the number of instances $T_n$ on the interval $(0,t]$, formally defined as
\[\nu(t):=\nu(t,\delta)=\sup\{n> 0: T_n:=T_n(\delta)\le t\}.\]
Then
\begin{equation}
\label{eq:nu_bound}
\E(||C_t||_1)\le 2a\E(\delta^{\nu(t)}), \quad t\ge 0.\end{equation}
\end{theorem}

\begin{proof} First, notice that by construction of the RLGL algorithm, $||C_t||_1$ in non-increasing in $t$.
Indeed, if, for example, a positive cash arrives to a node with a positive cash, then the total amount of positive cash remains the same, but if a positive cash arrives to a node with negative cash, then the positive and the negative cash cancel each other, and only the overshoot (positive or negative) remains in the node. Therefore, since $T_{\nu(t)}\le t$, we have
\[||C_t||_1\le ||C_{T_{\nu(t)}}||_1.\]
Moreover, $\nu(t)$ is a counting process completely determined by $\bT$. In other words, conditioned on $\bT$, $\nu(t)$ can be treated as a deterministic variable. Hence, using \eqref{eq:C_tv0} and \eqref{eq:alpha}, we obtain:
$$
\E(||C_t||_1) \le E(||C_{T_{\nu(t)}}||_1) =  2\E\left(d_{TV}(X_{T_{\nu(t)}},X'_{T_{\nu(t)}})\right) =
$$
$$
2\E\left(\E\left[d_{TV}(X_{T_{\nu(t)}},X'_{T_{\nu(t)}})|\bT\right]\right)
\le 2a\,\E(\delta^{\nu(t)}).
$$
\hfill \end{proof}

Theorem~\ref{th:T} implies that $\E(||C_t||_1)$ converges exponentially fast to zero if $\bG$, deterministic or stochastic, is such that it is possible to construct $\bT$, of which the corresponding $\nu(t)$ satisfies
\begin{equation}
\label{eq:exp_convergence_P}
\E(\delta^{\nu(t)})\le b\rho^t \quad \mbox{for some $\rho\in (0,1)$, $b>0$}.
\end{equation}

\begin{remark}
Note that, formally, \eqref{eq:exp_convergence_P} does not follow from the statement `{\it $\delta^{\nu(t)}\le b\rho^t$ with high probability}'. For example, \eqref{eq:exp_convergence_P} is violated if $\P(\delta^{\nu(t)}\le b\rho^t)=1-1/t$, and $\P(\nu(t)=0)=1/t$.
\end{remark}

Of course, sequence $\bT$ can be found only if $\bG$ allows convergence. As a very simple example of a poorly chosen $\bG$, assume that $P_{11}=0$, and $G_t=\{1\}$ for all $t\ge 1$, so that only node~1 receives green light. Then $||C_t||_1$ will stop decreasing after $t=1$, and we have $\bT=(T_1)$, a finite sequence that obviously violates \eqref{eq:exp_convergence_P}.
Furthermore, if the algorithm converges at finite time $t^*$, then $||C_t||_1=0$ for all $t>t^*$, and \eqref{eq:alpha} is formally satisfied with $\bT=(t^*+1,t^*+2,\ldots)$, and any $\delta\in (0,1)$. Also, observe that the sequence $\bT(\delta)$ is not unique, for example, the sequence $\bT(\delta^2)=(T_0(\delta),T_2(\delta),T_4(\delta),\ldots)$ also satisfies \eqref{eq:alpha}. Finding smaller $T_n$ for a given $\delta$ will yield better convergence guarantees.

Below in Sections~\ref{ssec:dobrushin}--\ref{ssec:geometric} we show how Theorem~\ref{th:T} is used to prove exponential convergence in natural  examples with arbitrary positive recurrent and aperiodic $P$.

\subsection{$P$ with Dobrushin coefficient smaller than one, generalized cyclic $\bG$}
\label{ssec:dobrushin}
\label{sec:dobrushin}

The Dobrushin coefficient is defined as
\begin{equation}
\label{eq:dobrushin}
\delta(P)=1-\min_{i,i'\in [N]}\sum_{j\in [N]}\min\{p_{ij},p_{i'j}\}.
\end{equation}
In this section we assume that $\delta(P)<1$.  This includes important examples such as PageRank, where $W$ is the transition matrix of a simple random walk on a (directed) graph, and $P$ is the Google matrix
$P = cW +\frac{(1-c)}{N}{\bf 1}{\bf 1}^T$, $c\in (0,1)$. In this case, $\delta(P)\le c$.

We consider a class of $\bG$ that gives green light to all nodes within a finite time $m$:
\begin{equation}
\label{eq:cycle}
\cup_{l=1}^m G_{mn+l} = [N].
\end{equation}  We call this  a `generalized cyclic' scheduling because it consists of cycles of length $m$ (such that all nodes receive green light during a cycle), but  the cycles do not need to be identical, and can have a random order as well. {The RLGL algorithm is agnostic to the value of $m$.
In the consensus literature \cite{dwork1988consensus} such scenario is referred to as {\it partial
synchrony} and in parallel and distributed computing literature \cite{BertsekasTsitsiklis1989}
such scenario is referred to as {\it partial asynchrony}.}

Our goal now is to establish exponential convergence of the RLGL algorithm in this specific case.  The result is formulated in the next theorem.

\begin{theorem}
\label{th:coupling1} If $P$ is such that $\delta(P)<1$ and $\bG$ satisfies \eqref{eq:cycle}, we have that:
\begin{itemize}
\item[-] if $\bG$ is deterministic then
\[||C_t||_1\le 2\delta(P)^{-1} \delta(P)^{\frac{t}{m}}, \quad t\ge 0;\]
\item[-] if $\bG$ is stochastic then
\[
\P\left((||C_t||_1)\le f(t)\delta(P)^{\frac{t}{m}}\right)=1-o(1), \quad t\to \infty.
\]
for any $f(t)>0$ such that $\lim_{t\to\infty}f(t)=\infty$.
\end{itemize}
If, in addition, $\inf_{t>0} |H_t{\bf 1}| >0$, it holds that:
\begin{itemize}
\item[-] if $\bG$ is deterministic then $||\hat \pi_t-\pi^*||_1 = O(\delta(P)^{\frac{t}{m}})$;
\item[-] if $\bG$ is stochastic then $||\hat \pi_t-\pi^*||_1 = O_{\P}(f(t)\delta(P)^{\frac{t}{m}})$ for any  $f(t)>0$ such that $\lim_{t\to\infty}f(t)=\infty$.
\end{itemize}
\end{theorem}

\begin{proof}
We will construct a sequence $\bT(\delta(P))$ that satisfies the conditions of Theorem~\ref{th:T}. Let $(\bX,\bX')$ be as in Definition~\ref{def:XX'}. It follows from the coupling inequality that
\begin{equation}
\label{eq:coupling_inequality}
\E(d_{TV}(X_t, X_t')| \bT) \le \P(\hat X_t\ne \hat X'_t|\bT), \quad t\ge 0,
\end{equation}
where,  conditioned on $\bT$, $(\hat \bX, \hat \bX')$ is a coupling of $(\bX, \bX')$. We will prove the theorem by constructing a coupling $(\hat \bX, \hat \bX')$  such that $\P(\hat X_t\ne \hat X'_t)$ decreases by factor $\delta(P)$ every $m$ steps.

Let us now construct a suitable coupling $(\hat \bX, \hat \bX')$. First, using the standard approach, set $\hat X_1 = X_1$ and $\hat X'_1 = X'_1$. For each $n\ge 0$, if $\hat X_{nm+1}= \hat X'_{nm+1}$ then we let the two processes continue together: $\hat X_{t}= \hat X'_{t}$, $t\ge nm+1$. If $\hat X_{nm+1}\ne \hat X'_{nm+1}$, then we need to construct a coupling on the interval $[nm+1,(n+1)m)$. We do this as follows. Let  $t_1,t_1' \in [nm+1,(n+1)m)$ be the times when, respectively $\hat \bX$ and $\hat \bX'$ receive green light for the first time in the cycle. It follows from \eqref{eq:cycle} that such $t_1, t_1'$ exist. Without loss of generality, assume that $t_1\le t_1'$. We will construct a coupling such that
\begin{equation}
\label{eq:coupling_t1}
\P(\hat X_{t'_1+1} = \hat X'_{t'_{1}+1}|\hat X_{t'_1} \ne \hat X'_{t'_{1}}, \hat X_{t'_1}=i, \hat X'_{t'_{1}}=i') = \sum_{j\in[N]}\min\{p_{ij}, p_{i'j}\}.\end{equation}
We start with coupling the transitions of $\hat \bX$ and $\hat \bX'$ at times  $t_1$  and $t_1'$ using a maximal coupling:
\begin{equation}
\label{eq:t1}
\P(\hat X_{t_1+1} = \hat X'_{t'_{1}+1}=j|\hat X_{t_1} =i, \hat X'_{t'_{1}} =i') = \min\{p_{ij}, p_{i'j}\},\quad i,j,k\in [N]. \end{equation}
Now, three cases are possible.
\begin{itemize}
\item If $t_1=t_1'$ then by \eqref{eq:t1}, \eqref{eq:coupling_t1} holds.
\item If $t_1< t_1'$, then we need to consider  two sub-cases.
\begin{itemize}
\item If $\hat \bX$ does not receive green light at $t_1<t\le t_1'$, then $\hat X_{t_1'+1}=\hat X_{t_1+1}$, and by \eqref{eq:t1}, \eqref{eq:coupling_t1} holds.
\item If $\hat \bX$ receives green light at some time $t_s\le t_1'$, $s=2,3,\ldots t_1'-t_1$, then the coupling \eqref{eq:t1} is canceled. Recall that the transition of $\hat \bX'$ in \eqref{eq:t1} was only planned but not yet realized by the time $t_s$ because $
\hat \bX'$ moves at $t=t'_1>t_s$ for the first time. Therefore, this `planned' transition does not need to occur. Instead, we create a new coupling
\begin{equation}
\label{eq:ts}
\P(\hat X_{t_s+1} = \hat X'_{t'_{1}+1}=j|\hat X_{t_s} =i, \hat X'_{t'_{1}} =i') = \min\{p_{ij}, p_{i'j}\},\quad i,j,k\in [N]. \end{equation}
Now, the procedure can be repeated. Specifically, if $\hat \bX$ does not receive the green light at $t_s<t\le t_1'$ then \eqref{eq:coupling_t1} follows from \eqref{eq:ts}. If $\hat \bX$ receives green light at some time $t_{\tilde s} \le t_1'$, $\tilde s=s+1,s+2,t_1'-t_1$, then we cancel coupling \eqref{eq:ts}, reset $s\leftarrow \tilde s$, and apply \eqref{eq:ts} again. This needs to be done at most $t_1'-t_1$ times as long as $\hat \bX$ keeps receiving green light at times $t_{s} \le t_1'$. By construction, for any $t_s\le t_1'$ when $\hat \bX$ receives green light for the last time before and including $t_1'$, \eqref{eq:ts} holds.
\end{itemize}
\end{itemize}
We conclude that for our constructed $\hat \bX$ and $\hat \bX'$, \eqref{eq:coupling_t1} holds, therefore, we obtain:
\begin{align}
\nonumber
\P&(\hat X_{(n+1)m+1}  \ne \hat X'_{(n+1)m+1}|\hat X_{nm+1}\ne \hat X'_{nm+1}) \le \P(\hat X_{t'_1+1} \ne \hat X'_{t'_{1}+1}|\hat X_{nm+1}\ne \hat X'_{nm+1})\\
\nonumber
&=1-\sum_{i,i'\in [N]} \P(\hat X_{t'_1+1} = \hat X'_{t'_1+1}|\hat X_{nm+1}\ne \hat X'_{nm+1}, \hat X_{t'_1} =i, \hat X'_{t'_1} =i')\\
\nonumber
&\qquad\qquad\qquad\qquad\qquad\qquad\qquad\qquad\qquad \times P(\hat X_{t'_1} =i, \hat X'_{t'_1} =i'|\hat X_{nm+1}\ne \hat X'_{nm+1})\\
\nonumber
&= 1- \sum_{i,i'\in [N]} \P(\hat X_{t'_1} =i, \hat X'_{t'_1} =i'|\hat X_{nm+1}\ne \hat X'_{nm+1})\sum_{j\in [N]}  \P(\hat X_{t'_1+1} = \hat X'_{t'_1+1}=j|\hat X_{t'_1} =i, \hat X'_{t'_1} =i')\\
\nonumber
& = 1-\sum_{i,i'\in [N]} \P(\hat X_{t'_1} =i, \hat X'_{t'_1} =i'|\hat X_{nm+1}\ne \hat X'_{nm+1})\sum_{j\in [N]}  \min\{p_{ij}, p_{i'j}\}\\
\label{eq:coupling1}
&\le 1- \inf_{i,i'\in [N]}\sum_{j\in [N]}  \min\{p_{ij}, p_{i'j}\}  = \delta(P),
\end{align}
where the second equality is due to the Markov property. Iteratively applying  \eqref{eq:coupling1}, we obtain that the deterministic sequence $\bT$ of time instants $T_n=T_n(\delta(P)) = nm+1$, $n\ge 1$, satisfies
\[\E(d_{TV}(X_{T_n},X'_{T_n})|\bT)=\E(d_{TV}(X_{T_n},X'_{T_n}))\le \P(\hat X_{T_n}\ne \hat X'_{T_n}) \le  (\delta(P))^n.\]
Thus, by Theorem~\ref{th:T} we have
\[\E(||C_t||_1)\le 2\delta(P)^{\lfloor\frac{t-1}{m}\rfloor}\le 2\delta(P)^{-1} \delta(P)^{\frac{t}{m}},\] and the result follows from Lemma~\ref{th:general} and Remark~\ref{remark:det}.
\end{proof}

\subsection{Cyclic $\bG$}
\label{ssec:irreducible}

In this section we obtain the exponential convergence for  deterministic cyclic $\bG$. Such deterministic cycling, or round robin, is a very common scheduling approach in computing systems. The main result is stated in the next theorem.

\begin{theorem}
\label{th:coupling2} Let $B_1,B_2,\ldots,B_m\subseteq [N]$ be such that $\cup_{l=1}^mB_l = [N]$, and $G_{nm+l} = B_l$, $n\ge 0$, $1\le l\le m$.
Denote by
\[P({\bf G}_m)=\prod_{l=1}^mP(G_l)\]
the transition probability matrix of one complete cycle. If,  for some $r$, the matrix $P({\bf G}_m)^r$ is Markov matrix
$($see \cite{seneta2006book}$)$, i.e.,
\begin{equation}
\label{eq:markov-cycle}
\left( \left(P({\bf G}_m)\right)^{r}\right)_{ij_0}> 0 \mbox{ for some  $j_0$ and all $i\in[N]$},
\end{equation}
then
 \[||C_t||_1=\E(||C_t||_1)\le 2[1-\eta^2]^{-1}(1-\eta^2)^{\frac{t}{rm}}, \quad t\ge 0,\]
where

\begin{equation}
\eta = \inf_{i}\left( \left(P({\bf G}_m)\right)^{r}\right)_{ij_0}.
\end{equation}

If, in addition, $\inf_{t>0} |H_t{\bf 1}| >0$, then  $||\hat \pi_t-\pi^*||_1 = O\left((1-\eta^2)^{\frac{t}{rm}}\right)$.

\end{theorem}

\begin{proof} Let $(\bX,\bX')$  be as in Definition~\ref{def:XX'}. We will again use the coupling inequality \eqref{eq:coupling_inequality}. The coupling $(\hat \bX, \hat \bX')$ is constructed in a standard way as follows. Set $\hat X_1 = X_1$ and $\hat X'_1 = X'_1$. Given ${\bf G}$, Markov chains  $\hat \bX$ and $\hat \bX'$ evolve  independently  until the fist time they meet at some time $s\ge 1$. After that $\hat \bX$ and $\hat \bX'$ continue together: $\hat X_{t}= \hat X'_{t}$, $t\ge s$.
Between the time instants $nm$ and $(n+1)m$, $n\ge 1$, the processes $\hat \bX$ and $\hat \bX'$ make a transition independently according to the transition matrix $P({\bf G}_m)$. Therefore, by \eqref{eq:markov-cycle}, for any $n\ge 0$ we have
\begin{align}
\nonumber
\P&(\hat X_{(n+1)rm}  \ne \hat X'_{(n+1)rm}|\hat X_{nrm}\ne \hat X'_{nrm}) \\
\nonumber
&= 1- \sum_{i,i',j\in [N]}  \P(\hat X_{nrm}=i, \hat X'_{nrm}=i', \hat X_{(n+1)rm} = \hat X_{(n+1)rm}' = j |\hat X_{nrm}\ne \hat X'_{nrm})\\
\nonumber
&\le 1- \sum_{i,i'\in [N]}  \P(\hat X_{nrm}=i, \hat X'_{nrm}=i', \hat X_{(n+1)rm} = \hat X_{(n+1)rm}' = j_0 |\hat X_{nrm}\ne \hat X'_{nrm})\\
&\le 1-\sum_{i,i'\in [N]}  \P(\hat X_{nrm}=i, \hat X'_{nrm}=i' |\hat X_{nrm}\ne \hat X'_{nrm}) \left(P({\bf G}_m)\right)_{ij_0}^{r}\left(P({\bf G}_m)\right)_{i'j_0}^{r}\le 1- \eta^2.
\label{eq:coupling2}
\end{align}
By \eqref{eq:coupling2} and \eqref{eq:coupling_inequality}, we have that
the deterministic sequence $\bT$ of time instants $T_n=nrm+1$, $n\ge 1$, which satisfies
\[
\E(d_{TV}(X_{T_n},X'_{T_n})|\bT)= \E(d_{TV}(X_{T_n},X'_{T_n}))\le \P(\hX_{T_n}\ne \hX'_{T_n})  \le 2(1-\eta^2)^n, n\ge 1.
\]
Thus, by Theorem~\ref{th:T} we have
\[\E(||C_t||_1)\le 2(1-\eta^2)^{\lfloor\frac{t-1}{rm}\rfloor}\le 2(1-\eta^2)^{-1}(1-\eta^2)^{\frac{t}{rm}},\]
and the result follows form Lemma~\ref{th:general} and the fact that $\bG$ is deterministic, see Remark~\ref{remark:det}.
\end{proof}

\bigskip

Note that \eqref{eq:markov-cycle} is a weaker condition than the irreducibility and aperiodicity of $P({\bf G}_m)$ because the latter requires \eqref{eq:markov-cycle} to hold for all $i,j_0\in [N]$. Moreover,  the order in the cycle is important because \eqref{eq:markov-cycle} may hold for one order of the green lights and not hold for some other order, even if the original matrix $P$ is irreducible and aperiodic. This is illustrated in the next example.

\begin{example}
\label{ex:counterexample2}
Continue Example~\ref{ex:counterexample} with $P$ as in \eqref{eq:counterexampleP}. Note that $P$ is aperiodic and irreducible, and
$\pi^*= (2/7, 1/7, 2/7, 2/7)$.

Assume that the green light is given to single nodes in some cyclic order. Note that in this case, if the last node to receive a green light has zero transition probability to itself, then the corresponding column in $P({\bf G_m})$ is null: the matrix will not be irreducible. First, we will give an example of a cycle that satisfies  \eqref{eq:markov-cycle}, and Theorem~\ref{th:coupling2} applies. Then, we will give an example where Theorem~\ref{th:coupling2} does not apply but where convergence may occur depending on the initial cash.

Consider the following cycle: ${\bf G_4} = (\{2\}, \{1\}, \{3\}, \{4\})$. Then we have
\[P({\bf G_4}) = \left(\begin{array}{cccc} 1/2&1/2&0&0\\1&0&0&0\\1&0&0&0\\1&0&0&0\end{array}\right),\]
and Theorem~\ref{th:coupling2} applies with $r=1$ and $\eta=1/2$. For arbitrary $M_0$, $H_t$ and $C_t$ evolve as follows:
\medskip

{\centerline{
{\small
\setlength{\tabcolsep}{1pt}
\resizebox{\columnwidth}{!}{%
\begin{tabular}{|c|c c c c|c c c c|c|}
\hline
$t$&&$H_t$&&&&$C_t$&&&$G_t$\\
\hline
1&$\left(M_{0,1}\right.,$&$ M_{0,2},$&$ M_{0,3},$&$\left. M_{0,4}\right)$&		$\left(M_{0,4}-M_{0,1}\right.,$&$ \frac{M_{0,1}}2-M_{0,2},$&$ \frac{M_{0,1}}2+M_{0,2}-M_{0,3},$&$\left. M_{0,3}-M_{0,4}\right)$&$\{2\}$\\
2&$\left(M_{0,4}\right.,$&$ \frac{M_{0,1}}2,$&$ M_{0,3},$&$ \left.M_{0,4}\right)$&		$\left(M_{0,4}-M_{0,1}\right.,$&$0  ,$&$M_{0,1}-M_{0,3},$&$\left.M_{0,3}-M_{0,4}\right)$&$\{1\}$\\
3&$\left(M_{0,4}\right.,$&$ \frac{M_{0,1}}2,$&$ M_{0,3},$&$ \left.M_{0,4}\right)$&		$\Bigl(0,$&$\frac{M_{0,4}-M_{0,1}}{2},$&$\frac{M_{0,1}+M_{0,4}}{2}-M_{0,3},$&$ M_{0,3}-M_{0,4}\Bigr)$&$\{3\}$\\
4&$\left(M_{0,4}\right.,$&$ \frac{M_{0,1}}2,$&$ \frac{M_{0,1}+M_{0,4}}{2},$&$ \left.M_{0,4}\right)$&	$\Big(0\Big.,$&$\frac{M_{0,4}-M_{0,1}}{2},$&$0,$&$\left.\frac{M_{0,1}-M_{0,4}}{2}\right)$&$\{4\}$ \\
5&$\Big( M_{0,4}\big.,$&$ \frac{M_{0,1}}2,$&$ \frac{M_{0,1}+M_{0,4}}{2}$&$ \left.\frac{M_{0,1}+M_{0,4}}{2}\right)$&		$\left(\frac{M_{0,1}-M_{0,4}}{2}\right.,$&$\frac{M_{0,4}-M_{0,1}}{2},$&$0,$&$0\Bigr)$&$\{2\}$ \\
6&$\Big(M_{0,4}\big.,$&$ \frac{M_{0,1}}2,$&$ \frac{M_{0,1}+M_{0,4}}{2},$&$ \left.\frac{M_{0,1}+M_{0,4}}{2}\right)$&	$\left(\frac{M_{0,1}-M_{0,4}}{2}\right.,$&$0  ,$&$\frac{M_{0,4}-M_{0,1}}{2},$&$0\Bigr)$&$\{1\}$ \\
7&$\left(\frac{M_{0,1}+M_{0,4}}{2}\right.,$&$ \frac{M_{0,4}}2,$&$ \frac{M_{0,1}+M_{0,4}}{2},$&$ \left.\frac{M_{0,1}+M_{0,4}}{2}\right)$&	$\Bigl(0,$&$\frac{M_{0,1}-M_{0,4}}{4},$&$\frac{M_{0,4}-M_{0,1}}{4},$&$0\Bigr)$&$\{3\}$ \\
8&$\left(\frac{M_{0,1}+M_{0,4}}{2}\right.,$&$ \frac{M_{0,4}}2,$&$ \frac{M_{0,1}+3M_{0,4}}{4},$&$ \left.\frac{M_{0,1}+M_{0,4}}{2}\right)$&	$\Bigl(0,$&$\frac{M_{0,1}-M_{0,4}}{4},$&$0,$&$\left.\frac{M_{0,4}-M_{0,1}}{4}\right)$&$\{4\}$ \\
9&$\left(\frac{M_{0,1}+M_{0,4}}{2}\right.,$&$ \frac{M_{0,4}}2,$&$ \frac{M_{0,1}+3M_{0,4}}{4},$&$ \left.\frac{M_{0,1}+3M_{0,4}}{4}\right)$&	$\left(\frac{M_{0,4}-M_{0,1}}{4}\right.,$&$\frac{M_{0,1}-M_{0,4}}{4},$&$0,$&$0\Bigr)$&$\{2\}$ \\
10&$\left(\frac{M_{0,1}+M_{0,4}}{2}\right.,$&$ \frac{M_{0,1}+M_{0,4}}{4},$&$ \frac{M_{0,1}+3M_{0,4}}{4},$&$ \left.\frac{M_{0,1}+3M_{0,4}}{4}\right)$&	$\left(\frac{M_{0,4}-M_{0,1}}{4}\right.,$&$0,$&$\frac{M_{0,1}-M_{0,4}}{4},$&$0\Bigr)$&$\{1\}$ \\
11&$\left(\frac{M_{0,1}+3M_{0,4}}{4}\right.,$&$ \frac{M_{0,1}+M_{0,4}}{4},$&$ \frac{M_{0,1}+3M_{0,4}}{4},$&$ \left.\frac{M_{0,1}+3M_{0,4}}{4}\right)$&	$\Bigl(0,$&$\frac{M_{0,4}-M_{0,1}}{8},$&$\frac{M_{0,1}-M_{0,4}}{8},$&$0\Bigr)$&$\{3\}$ \\
\hline
\end{tabular}%
}
}
}
}
\medskip

One can verify that if $M_{0,1}\ne 0$ or $M_{0,4}\ne 0$, in particular, if the initial cash is uniformly distributed, then $C_t$ tends to zero, $H_t\bf 1$ is lower bounded by a strictly positive value, and RLGL converges.


Consider now the following cycle: ${\bf G_4} = (\{1\}, \{2\}, \{4\}, \{3\})$. Then we have
\[P({\bf G_4}) = \left(\begin{array}{rrrr} 0&0&0&1\\0&0&0&1\\0&0&0&1\\1&0&0&0\end{array}\right).\]
This matrix is periodic after one iteration, and \eqref{eq:markov-cycle} fails. The next table shows  that RLGL may still converge for certain values of $M_0$:
\medskip

{\centerline{
{\small
\setlength{\tabcolsep}{3pt}
\begin{tabular}{|c|c c c c|c c c c|c|}
\hline
$t$&&$H_t$&&&&$C_t$&&&$G_t$\\
\hline
1&$(M_{0,1},$&$ M_{0,2},$&$ M_{0,3},$&$ M_{0,4})$&$(M_{0,4}-M_{0,1},$&$ 0.5 M_{0,1}-M_{0,2},$&$ 0.5 M_{0,1}+M_{0,2}-M_{0,3},$&$ M_{0,3}-M_{0,4})$&$\{1\}$\\
2&$(M_{0,4},$&$ M_{0,2},$&$ M_{0,3},$&$ M_{0,4})$&$(0,$&$0.5M_{0,4}-M_{0,2},$&$M_{0,2}+0.5M_{0,4}-M_{0,3},$&$M_{0,3}-M_{0,4})$&$\{2\}$\\
3&$(M_{0,4},$&$ 0.5M_{0,4},$&$ M_{0,3},$&$ M_{0,4})$&$(0,$&$0,$&$M_{0,4}-M_{0,3},$&$M_{0,3}-M_{0,4})$&$\{4\}$\\
4&$(M_{0,4},$&$ 0.5M_{0,4},$&$ M_{0,3},$&$ M_{0,3})$&$(M_{0,3}-M_{0,4},$&$0,$&$M_{0,4}-M_{0,3},$&$0)$&$\{3\}$ \\
5&$(M_{0,4},$&$ 0.5M_{0,4},$&$ M_{0,4},$&$ M_{0,3})$&$(M_{0,3}-M_{0,4},$&$0,$&$0,$&$M_{0,4}-M_{0,3})$&$\{1\}$ \\
6&$(M_{0,3},$&$ 0.5M_{0,4},$&$ M_{0,4},$&$ M_{0,3})$&$(0,$&$0.5M_{0,3}-0.5M_{0,4},$&$0.5M_{0,3}-0.5M_{0,4},$&$M_{0,4}-M_{0,3})$&$\{2\}$ \\
7&$(M_{0,3},$&$ 0.5M_{0,3},$&$ M_{0,4},$&$ M_{0,3})$&$(0,$&$0,$&$M_{0,3}-M_{0,4},$&$M_{0,4}-M_{0,3})$&$\{4\}$ \\
\hline
\end{tabular}
}
}
}

 We see that the expressions for $H_t$ and $C_t$ at $t=3$ and $t=7$ are exactly symmetrical with respect to $M_{0,3}$ and $M_{0,4}$, so these steps will repeat indefinitely. If $M_{0,3}=M_{0,4}=0$, then we have $H_t=C_t={\bf 0}$ at $t=3$ just as in Example~\ref{ex:counterexample}, so the algorithm finds the trivial solution $H_t={\bf 0}$, and $\hat \pi_t$ is undefined. If $M_{0,3} = M_{0,4}>0$ then at $t=3$ the algorithm converges to the correct solution $\hat \pi_3=\pi^*$. In all other cases the algorithm does not converge. In terms of the coupling $(\hat \bX, \hat \bX')$, with positive probability, $\hat \bX$ and $\hat \bX'$ will never meet.

\end{example}

\subsection{Random $\bG$}
\label{ssec:geometric}

Assume here that $P$ is irreducible and aperiodic, and green light is given to only one state, randomly chosen from $[N]$ at each time step, independently of anything else. This is the basic random scheduling technique that immediately leads to fully distributed computations. The next theorem states exponential convergence of the RLGL algorithm in this case.

\begin{theorem}
\label{th:coupling_random}  Let  $G_t=\{U_t\}$, where $U_1,U_2,\ldots$ are independent random variables with uniform distribution on $\{1,2,\ldots,N\}$, $N>2$. Assume that $P$ is irreducible and aperiodic, and $\eta\in(0,1)$ $r\ge 1$ are such that
\begin{equation}
\label{eq:recurrent}
(P^{r'})_{ij}>\eta \quad \mbox{for all $r'>r$ and $i,j\in [N]$.}
\end{equation}
Then for any $f(t)> 0$ such that $\lim_{t\to\infty} f(t) =\infty$ it holds that
\[
\P\left(||C_t||_1\le f(t)  e^{-a t}\right) = 1-o(1), \quad  \mbox {as $t\to\infty$},
\]
where
\begin{align*}
a &= \sup_{\beta\in \left(0,(Nr)^{-1}\right)} \min\{\beta\,|\log\left(1-\tilde\eta\right)|, \beta J(\beta)\},\\
\tilde\eta &= \left(\frac{1}{2}\right)^{r-1}\eta,\\
J(\beta) &= \beta^{-1}(1- 2r\beta)\,\log\left(\frac{N(1-2r\beta)}{N-2}\right) + 2r \log\left({Nr\beta}\right).
\end{align*}
If, in addition, $\inf_{t>0} |H_t{\bf 1}| >0$, then $||\hat \pi_t-\pi^*||_1=O_{\P}\left(f(t)e^{-at}\right)$ as $t\to\infty$.
\end{theorem}

\begin{proof} The theorem is proved in two steps. First, we find a sequence $\bT$ that satisfies \eqref{eq:alpha} with $\delta=1-\tilde\eta$ and apply Theorem~\ref{th:T}. Second, we derive a lower bound for the convergence rate of $\E(\left(1-\tilde\eta\right)^{\nu(t)})$.

As before, we use the coupling technique. Let $\bX$ and $\bX'$ be as in Definition~\ref{def:XX'}. Without loss of generality, assume that, conditioned on $\bG$, $\bX$ and $\bX'$ are independent, and denote by $t_0>0$ the first time instant when they meet. Set $\hat X_t= X_t$, $\hat X'_t= X'_t$ for $t\le t_0$  and $\hat X_t= \hat X'_t=X_t$ for $t>t_0$. Again, by \eqref{eq:C_tv0} and \eqref{eq:coupling_inequality}, it is sufficient to construct $\bT$ such that $\hat \bX$ and $\hat \bX'$ meet between $T_{n-1}$ and $T_n$ with positive probability. Since $P$ is irreducible and aperiodic, there exist $\eta\in (0,1)$ and $r\ge 1$ that satisfy \eqref{eq:recurrent}. Define $\bT$ as follows:
\begin{equation}\label{eq:Delta}T_0=1,\quad T_n=T_{n-1}+\Delta_n, \quad n\ge 1,\end{equation}
where $\Delta_n=\xi_1+\xi_2+\cdots\xi_{2r}$, and $\xi_1,\xi_2,\ldots,\xi_{2r}$ are independent geometric random variables with parameter $2/N$. As long as $\hbX$ or $\hbX'$ have not met, set $\xi_i$ to be the time until either $\hbX$ or $\hbX'$ make a transition; this time indeed has geometric distribution with parameter $2/N$. Then, by the time $T_n$ two cases are possible: 1) $\hbX$ and $\hbX'$ meet before $T_n$,  or 2) by the time $T_n$, $\hbX$ and $\hbX'$ together make $2r$ transitions. In the latter case, with probability $2\cdot (1/2)^r= (1/2)^{r-1}$ the last $r$ transitions belong to the same process, say, $\hbX$, without loss of generality. In that case, during these last $r$ transitions, $\hbX'$ remains at some state $j\in[N]$.  Furthermore, by \eqref{eq:recurrent}, $\hbX$ will reach $j$ by the time $T_n$ with probability at least $\eta$. Hence, we write:
\begin{align}
\nonumber \P&(\hX_{T_n}  \ne \hX'_{T_n}|\hX_{T_{n-1}}  \ne \hX'_{T_{n-1}}, \bT)\\
&= 1- \sum_{j\in[N]}\P\left(\hX_{T_n} = \hX'_{T_n}=j|\hat X_{T_{n-1}}  \ne \hat X'_{T_{n-1}},\bT\right)\le  1-\frac{1}{2^{r-1}}\eta = 1-\tilde\eta.
\label{eq:rsteps}
\end{align}
Then, by \eqref{eq:rsteps} and \eqref{eq:coupling_inequality}, we get
\[
\E(d_{TV}(X_{T_n},X'_{T_n})|\bT)\le \P(\hX_{T_n}  \ne \hX'_{T_n}|\bT)\le (1-\tilde\eta)^n, n\ge 1,
\]
and Theorem~\ref{th:T} gives
\begin{equation}
\label{eq:randomC}
\E(||C_t||_1)\le 2\E((1-\tilde\eta)^{\nu(t)}),\end{equation}
where $\nu(t)=\sup_{n\ge 0}\{n: T_n\le t\}$, $t\ge 0$.

It remains to evaluate \[\E((1-\tilde\eta)^{\nu(t)})=\E(e^{\log(1-\tilde\eta)\,{\nu(t)}})=\E(e^{-\alpha \nu(t)}),\] where $\alpha=-\log(1-\tilde\eta)$.  For any $\beta>0$, we obtain
\begin{align}
\label{eq:ld1}
\E(e^{-\alpha \nu(t)}) & \le \E(e^{-\alpha \beta t}{\bf 1}\{\nu(t) \ge \beta t\}) + \E({\bf 1}\{\nu(t) < \beta t\})\le e^{-\alpha \beta t} + \P(\nu(t)<\beta t).
\end{align}
Using the fact that $\nu(t)$ is a discrete-time renewal process with inter-arrival times distributed as $\Delta_n$, we can apply the large deviation results (see e.g. \cite{Lefevere2011large_deviations}) to evaluate the last term in \eqref{eq:ld1}. We present the derivation below for completeness. First, we write:
\begin{align}
\nonumber
\P(\nu(t)< \beta t)& =  \P(\nu(t) \le \lfloor \beta t \rfloor) = \P(T_{\lfloor \beta t \rfloor} \ge t) \\
\label{eq:ld2}
&= \P\left(T_{\lfloor \beta t \rfloor} \ge \frac{t}{\lfloor \beta t \rfloor}\, \lfloor \beta t \rfloor\right) \le P\left(T_{\lfloor \beta t \rfloor} \ge \beta^{-1}\, \lfloor \beta t \rfloor\right).
\end{align}
Recall that $\Delta_n$ is a sum of $2r$ independent geometrically distributed random variables with parameter $2/N$, so $\Delta_n$ has finite exponential moments and  mean $Nr$. Therefore, by the Cram\'er's bound (see e.g. Theorem~2.19 \cite{vdHofstad}), we have that the last expression in \eqref{eq:ld1} decreases exponentially in $\lfloor \beta t \rfloor$ when ${\beta}^{-1} > Nr$. Specifically, for any $\beta \in (0,(Nr)^{-1})$ we get
\begin{align}
\label{eq:ld3}
\P\left(T_{\lfloor \beta t \rfloor} \ge {\beta}^{-1}\, \lfloor \beta t \rfloor\right)\le  e^{- J(\beta) \lfloor \beta t \rfloor} \le e^{J(\beta)} e^{- \beta J(\beta)\,t},
\end{align}
where
\begin{align*}
J(\beta) &= \sup_{s\ge 0} \left\{ \beta^{-1}\,s - \log\E\left(e^{s\Delta_1}\right)\right\} = \beta^{-1}(1- 2r\beta)\,\log\left(\frac{N(1-2r\beta)}{N-2}\right) + 2r \log\left({Nr\beta}\right).
\end{align*}
From \eqref{eq:ld1} -- \eqref{eq:ld3} we have that
\[ \E(e^{-\alpha \nu(t)})  \le  e^{-\beta \alpha\, t} + e^{J(\beta)} e^{- \beta J(\beta)\,t}\le \left(1+ e^{J(\beta)}\right)e^{-\beta \min\{\alpha,J(\beta)\}\,t}.\]
We now choose $\beta\in(0, (Nr)^{-1})$ to maximize the negative exponent in the last expression. It is easy to check that $\beta J(\beta)>0$ for $\beta\in(0, (Nr)^{-1})$, so that we get \[a=\sup_{\beta\in(0, (Nr)^{-1})}\beta\min\{\alpha, J(\beta)\}>0.\]
The result now follows form (\ref{eq:randomC}) and Lemma~\ref{th:general}.
\end{proof}

\begin{remark} Note that our estimate $a$ of the convergence rate is smaller than intuitive value $\log(1-\eta)Nr$, There are two reasons for this. First, even if the processes $\bX$ is guaranteed to make at least $r$ transitions and reach any state with probability at least $\eta$, it may occur (as in Example~\ref{ex:counterexample2}) that $\bX'$ cannot be at the same state due to the realized sequence of green lights. Therefore, we must take into account the probability that a sequence of green lights allows $\bX'$ and $\bX$ to be in the same state. In the proof we used  a very crude lower bound for this probability, $(1/2)^{r-1}$. Nevertheless, even more accurate bounds must take into account `unfortunate' sequences of green lights, which will lower the estimation of the convergence rate. Second, the exponential convergence of  \eqref{eq:ld1}  requires that $\beta<(Nr)^{-1}$. This loss in convergence rate follows from the large deviations results, and arises from the variability of the random scheduling.
\end{remark}

Let us compare the performance  of RLGL with random $\bG$ with the method of power iteration.
Using Jensen's inequality, we can write
\begin{align*}
   \E(||C_t||)\ge ||\E(C_t)|| = \left|\left|(\pi_1-\pi_1')\prod_{k=1}^{t-1}\E(P(G_k))\right|\right|= \left|\left|(\pi_1-\pi_1')\Tilde{P}^{t-1}\right|\right|,
\end{align*}
where $\Tilde{P}=\frac{1}{N}\sum_{i=1}^N P({i})=(1-\frac{1}{N})I + \frac{1}{N} P$. We have the
following simple relation between the eigenvalues of $\tilde{P}$ and $P$:
$$
\lambda(\tilde{P})=1-\frac{1}{N}(1-\lambda(P)).
$$
It is natural to compare one power iteration with $N$ updates of RLGL with random $\bG$.
We have the following inequality
$$
\lambda \le \left(1-\frac{1}{N}(1-\lambda)\right)^N.
$$

{
This indicates that the performance of RLGL with random $\bG$ cannot be generically better than that of power iterations. }

We conclude this section with a remark about ordinal ranking based on approximate and exact stationary probability
distributions. In some applications such as web ranking the ordering given by the stationary probability distribution is more
important than the probabilities themselves. The relation between the numerical precision and the ordinal ranking is a difficult
topic with many non intuitive counterexamples \cite{Wills2007rank,Ipsen2009}. The exponential convergence has an important implication for ordinal ranking. Suppose
$$||\hat{\pi}_t-\pi^*||_1\le \alpha||\hat{\pi}_{t-1}-\pi^*||_1,\quad \alpha<1,$$
which corresponds to the cases described in Theorems~\ref{th:coupling1}~and~\ref{th:coupling2}. We can write
$$||\hat{\pi}_t-\pi^*||_1\le \alpha||\hat{\pi}_{t-1}-\hat{\pi}_{t}||_1+ \alpha||\hat{\pi}_{t}-\pi^*||_1,$$
and hence
$$||\hat{\pi}_t-\pi^*||_1\le \frac{\alpha}{1-\alpha}||\hat{\pi}_{t-1}-\hat{\pi}_{t}||_1.$$
Then, using the above inequality and Theorem 3.1 from \cite{Ipsen2009}, we draw a conclusion on the  ranking  between elements $i$ and $j$ using the residual at time step $t$. Specifically, from
$$\hat{\pi}_{t,i} > \hat{\pi}_{t,j} + \frac{\alpha}{1-\alpha}||\hat{\pi}_{t-1}-\hat{\pi}_{t}||_1$$
we are able to conclude that $ \pi^*_i>\pi^*_j.$

Summarizing the results of Sections~\ref{sec:dobrushin} -- \ref{ssec:geometric}, we conclude that RLGL algorithm converges exponentially in several natural scenarios, when the scheduling of $\bG$ makes no use of the structure $P$ or the cash distribution. We could evaluate the convergence rate more precisely by using a stronger coupling and making the $T_n$'s as small as possible. Yet, with such simple scheduling of green light we cannot expect a great improvement of convergence rate compared to the power iterations. The real potential of the RLGL algorithm is in intelligent  control of $\bG$ for faster convergence. We will demonstrate this below in Sections~\ref{sec:two_blocks}~and~\ref{sec:three_blocks}.

\section{Convergence faster than power iterations}
\label{sec:two_blocks}

In Section~\ref{sec:exp_convergence} the scheduling for green light was agnostic to the amount of cash in nodes. It is natural to consider a broader class of scheduling strategies, which do take the cash values into account for more efficient depletion. In this section we will start with a very simple model, i.e., the mean-field stochastic block model, and consider a schedule that always gives green light to one of the blocks. We will show that this results in faster convergence than that of power iterations.

\subsection{The mean-field model}


The Stochastic Block Model (SBM) \cite{SBM_first_ref_1,SBM_first_ref_2} is the simplest random graph
model with clustered structure. In its basic form, SBM has two clusters $B_1$ and $B_2$ with sizes $N_1$
and $N_2$. The nodes inside a cluster are connected with probability $p$ and the nodes across
different clusters are connected with probability $q$. In the following we shall consider the
Mean-Field SBM in which we replace the original graph by a weighted graph where
the edge weight corresponds to the probability of an edge. Furthermore, we will use the variable $n$ to parametrize the sizes of the clusters. Specifically, set $N_2=n$ and $N_1=Kn$, with $K \ge 1$. (We understand that $Kn$ stands for $\lfloor Kn \rfloor$ but write $Kn$ for brevity.)
Thus, the adjacency and transition probability matrices are, respectively,
\begin{align*}
A &= \left[\begin{array}{cc}
p J_{Kn,Kn} & q J_{Kn,n}\\
q J_{n,Kn} & p J_{n,n}
\end{array}\right],\\
P &= \left[\begin{array}{cc}
\frac{p}{pKn+qn} J_{Kn,Kn} & \frac{q}{pKn+qn} J_{Kn,n}\\
\frac{q}{qKn+pn} J_{n,Kn} & \frac{p}{qKn+pn} J_{n,n}
\end{array}\right],
\end{align*}
where $J_{a,b}$ is the matrix of ones with size $a \times b$.

It is easy to see that the second largest eigenvalue
of $P$ equals to
\[
\lambda_2 = \frac{(p^2-q^2)K}{(pK+q)(qK+p)}.
\]
It is well known that the convergence of power iterations is governed by the powers of $\lambda_2$.

Recall from \eqref{eq:Cpi-pi'} that
\[
C_t = \pi'_t - \pi_t,
\]
where $\pi_t = M_0 \prod_{\tau=1}^{t-1}P(G_\tau)$, $\pi'_t = M_0 P \prod_{\tau=1}^{t-1}P(G_\tau)$.
As before, we are interested in the rate of convergence of $||C_t||_1$ to zero as $t\to\infty$.

\subsection{Computation costs of one-block policies}
\label{ssec:two-block-costs}

Let us consider a policy when the green light is given to all nodes in one bock, i.e., $G_t=B_1$ or $G_t=B_2$ for all $t\ge 1$. Then all nodes in the same block are identical, and so we have
\begin{align}
\label{eq:pt}
\pi_t &= [p_{t,1}\onevec_{nK}^T \ p_{t,2}\onevec_{n}^T],\\
\label{eq:pt'}
\pi'_t &= [p'_{t,1}\onevec_{nK}^T \ p'_{t,2}\onevec_{n}^T],
\end{align}
where $\onevec_{a}$ is a column vector of ones of length $a$. We will now investigate the convergence rate and the computational costs of such policies.

If we move the cash from the first block, i.e., $G_t=B_1$, then we can write  $\pi_{t+1}$ and $\pi'_{t+1}$ as follows:
\begin{align}
\label{eq:b1pt}
\pi_{t+1} & =  \left[\left(\frac{p K}{pK + q}\ p_{t,1}\right) \onevec_{Kn} \quad \left(\frac{q K}{pK + q} \ p_{t,1} + p_{t,2}\right)\onevec_n\right], \\
\label{eq:b1pt'}
\pi'_{t+1} & =  \left[\left(\frac{p K}{pK + q}\ p'_{t,1}\right) \onevec_{Kn} \quad \left(\frac{q K}{pK + q} \ p'_{t,1} + p'_{t,2}\right)\onevec_n\right].
\end{align}

Since the total cash equals zero,
\[
Kn p_{t,1} + n p_{t,2} - Kn p'_{t,1} - n p'_{t,2} = 0,
\]
we obtain an explicit relation between the
cash of $B_1$ and the cash of $B_2$:
\begin{equation}
\label{eq:P1P2relat}
K (p_{t,1}-p'_{t,1}) = p'_{t,2} - p_{t,2}.
\end{equation}
Therefore, if $G_t=B_1$, using \eqref{eq:b1pt}, \eqref{eq:b1pt'}, \eqref{eq:P1P2relat}, yields:
\begin{align*}
||C_{t+1}||_1&= Kn |p_{t+1,1}-p'_{t+1,1}| + n |p_{t+1,2} - p'_{t+1,2}|\\
&= Kn \frac{p K}{pK + q} |p_{t,1}-p'_{t,1}|
+ n \left|\frac{q K}{pK + q}(p_{t,1}-p'_{t,1}) + (p_{t,2}-p'_{t,2})\right|\\
&=
Kn \frac{p K}{pK + q} |p_{t,1}-p'_{t,1}|
+ n \left|\frac{q K}{pK + q}(p_{t,1}-p'_{t,1}) - K(p_{t,1}-p'_{t,1})\right|\\
&
= Kn \frac{p K}{pK + q} |p_{t,1}-p'_{t,1}| +
Kn \left(1-\frac{q }{pK + q}\right) |p_{t,1}-p'_{t,1}|\\
&= 2 Kn \frac{pK}{pK + q} |p_{t,1}-p'_{t,1}|.
\end{align*}
Now, again using (\ref{eq:P1P2relat}), we derive that
$$
||C_{t}||_1 = Kn |p_{t,1}-p'_{t,1}| + n |p_{t,2}-p'_{t,2}|
= 2 Kn |p_{t,1}-p'_{t,1}|
$$
to finally obtain
\begin{equation}
\label{eq:CG1}
||C_{t+1}||_1  = \frac{pK}{pK + q} ||C_{t}||_1, \quad \mbox{if $G_t= B_1$}.
\end{equation}

Similarly, if we give green light to the second block, $G_t=B_2$, then:
\begin{align*}
\pi_{t+1} & = \left[\left(\frac{q}{qK + p}\ p_{t,2} + p_{t,1}\right)\onevec_{Kn}\; \left(\frac{p}{qK + p}\ p_{t,2}\right)\onevec_{n}\right] \\
\pi'_{t+1} & = \left[\left(\frac{q}{qK + p}\ p'_{t,2} + p'_{t,1}\right)\onevec_{Kn}\; \left(\frac{p}{qK + p}\ p'_{t,2}\right)\onevec_{n}\right].
\end{align*}
Again using the relation (\ref{eq:P1P2relat}), we obtain
\begin{equation}
\label{eq:CG2}
||C_{t+1}||_1  = \frac{p}{qK + p} ||C_{t}||_1, \quad \mbox{if $G_t= B_2$}.
\end{equation}
When $K>1$, the factor in (\ref{eq:CG2}) is smaller than the one in (\ref{eq:CG1}).
We conclude that under any schedule that gives green light to one block at a time,
$||C_{t}||_1$ converges exponentially to zero, and the convergence rate is the highest
if $G_t=B_2$ for all $t\ge 1$.

Next, let us compare the computational costs of RLGL with $G_t=B_2$ and PI assuming that each of the two algorithms stops after reaching
\[
\frac{1}{2}||C_t||_1 = d_{TV}(\pi_t,\pi'_t) \le \varepsilon.
\]
We set the cost of each operation with block $B$ to be equal to the volume of this block
${\rm vol}(B)=\sum_{i \in B}d_i$. We will be especially interested in large graphs, $n\to\infty$, in the regime when $K:=K(n)\to\infty$ and $q:=q(n)$, $p:=p(n)$ are such that $qK/p \to 0$ as $n\to\infty$. This corresponds to the interesting case when the system is nearly completely decomposable and the sizes of the blocks are unbalanced, which is hard to handle by power iterations \cite{stewart1991numerical}.
For other regimes, similar analysis can be easily performed, and we will leave it out for brevity.

If $G_t=B_2$ for all $t\ge 1$, then the cost of one step of RLGL is $(pn+qKn)n$, and we can estimate the total cost of the RLGL method as follows:
\begin{align*}
{\rm Cost(RLGL)} &= \left(\log_\frac{p}{qK+p} \varepsilon\right) \ (pn+qKn)n\\
&= \frac{\log(1/\varepsilon)}{\log(1+\frac{qK}{p})} \ pn^2 \left(1+\frac{qK}{p}\right)
= \log(1/\epsilon) \frac{p^2n^2}{qK}(1+o(1))\quad \mbox{as $n\to\infty$}.
\end{align*}

In order to evaluate the costs of PI, we first write:
\[
\lambda_2 =  \frac{(p^2-q^2)K}{(pK+q)(qK+p)} = 1 - \frac{pqK^2 + 2q^2K+pq}{p^2K+pqK^2 + q^2K +pq}
= 1 - \frac{qK}{p}\left(1+o(1)\right).
\]
Since $\lambda_2<\frac{p}{qK + p}=1-\frac{qK}{p+qk}$, each step $t$ of PI results in more cash being depleted than a step of RLGL with $G_t=B_2$. This is logical because PI is equivalent to $G_t=B_1\cup B_2$. However, one step of PI also involves higher computational costs:
\begin{align*}
{\rm Cost(PI)} &= \log_{\lambda_2} \varepsilon \ [(pn+qKn)n + (pKn+qn)Kn]\\
&= \frac{\log(\epsilon)}{\log(\lambda_2)}\ pn^2K^2\ (1+o(1))= \log(1/\epsilon) \, \frac{p^2n^2}{qK} K^2 \ (1+o(1)).
\end{align*}
Thus, in the asymptotic regime under consideration, the cost of the PI method is $K^2$ times larger than the cost of the RLGL method.

Interestingly, if green light is given to $B_1$ instead of $B_2$ at any time $t$, then not only the convergence rate of $||C_t||_1$ is lower by \eqref{eq:CG1} and \eqref{eq:CG2}, but also the computational costs of each iteration are higher due to the higher volume of $B_1$. We conclude that in the presented asymptotic regime, among the scheduling strategies that give either red or green light to an entire block, the schedule $G_t=B_2$, $t\ge 1$, has the lowest computational cost.

\subsection{Greedy green-light scheduling}

In this section we will investigate whether it could be optimal to give green light to a fraction of nodes in a block.
Assume that at the beginning of step $t$ all nodes in the same block are again identical, so that $\pi_t$, $\pi'_t$ are given by \eqref{eq:pt}, \eqref{eq:pt'}.
Consider a schedule that at time $t$ gives green light to  the first $\alpha_1Kn$ nodes of $B_1$ and  the first $\alpha_2 n$ nodes of $B_2$ for some $\alpha_1,\alpha_2\in [0,1]$. Then by similar calculations as in Section~\ref{ssec:two-block-costs} we have

\begin{align*}
\pi_{t+1}  = &\left[\left(\frac{\alpha_1 pK}{pK + q}\ p_{t,1} + \frac{\alpha_2q}{qK + p} \ p_{t,2}\right) \onevec_{\alpha_1 Kn} \;
\left(p_{t,1}+\frac{\alpha_1 pK}{pK + q}\ p_{t,1} + \frac{\alpha_2q}{qK + p} \ p_{t,2}\right) \onevec_{(1-\alpha_1) Kn}\right.\\
&\left.\left(\frac{\alpha_1 qK}{pK + q}\ p_{t,1} + \frac{\alpha_2 p}{qK + p} \ p_{t,2} \right) \onevec_{\alpha_2 n} \;
\left(p_{t,2}+\frac{\alpha_1 qK}{pK + q}\ p_{t,1} + \frac{\alpha_2 p}{qK + p}\ p_{t,2}\right) \onevec_{(1-\alpha_2) n}\right],
\end{align*}
and an identical expression holds for $\pi'_{t+1}$. Invoking (\ref{eq:P1P2relat}), we get
\begin{align*}
||C_{t+1}||_1 &= \alpha_1Kn\left|\frac{\alpha_1 pK}{pK+q} - \frac{\alpha_2q K }{qK + p}\right|\,|p_{t,1}-p'_{t,1}|\\
& + (1-\alpha_1) Kn \left(1+ \frac{\alpha_1 pK}{pK+q} - \frac{\alpha_2qK}{qK + p}\right) |p_{t,1} - p'_{t,1}|\\
&+  \alpha_2 K n \left|\frac{\alpha_1 q}{pK+q}-\frac{\alpha_2 p}{qK + p}  \right| |p_{t,1} - p'_{t,1}|\\
&+ (1-\alpha_2) Kn \left(1 - \frac{\alpha_1 q}{pK+q} + \frac{\alpha_2 p }{qK + p} \right) |p_{t,1} - p'_{t,1}|\\
&:=Kn \ |p_{t,1} - p'_{t,1}|\  A(\alpha_1,\alpha_2) = \frac{A(\alpha_1,\alpha_2)}{2} \ ||C_{t}||_1.
\end{align*}
Consider again our regime of interest when $K:=K(n)\to \infty$ and $Kq/p: = K(n)q(n)/p(n)\to 0$ as $n\to\infty$. Looking at the terms with the absolute values, we notice that for any $\alpha_1,\alpha_2>0$ we have
\[\lim_{n\to\infty}\frac{\alpha_1 pK}{pK+q} = \alpha_1,\quad \lim_{n\to\infty}\frac{\alpha_2q K }{qK + p} =0,\]
and
\[\lim_{n\to\infty}\frac{\alpha_1 q}{pK+q} =0, \quad \lim_{n\to\infty}\frac{\alpha_2 p}{qK + p} =\alpha_2.\]
Therefore, for large enough $n$ we obtain
\begin{align*}
A(\alpha_1,\alpha_2) &=
\alpha_1\left(\frac{\alpha_1 pK}{pK+q} - \frac{\alpha_2q K }{qK + p}\right) + (1-\alpha_1) \left(1+ \frac{\alpha_1 pK}{pK+q} - \frac{\alpha_2qK}{qK + p}\right)\\
&+  \alpha_2 \left(\frac{\alpha_2 p}{qK + p} - \frac{\alpha_1 q}{pK+q} \right)+ (1-\alpha_2) \left(1 - \frac{\alpha_1 q}{pK+q} + \frac{\alpha_2 p }{qK + p} \right)\\
& = 2 - \alpha_1 - \alpha_2 + \frac{\alpha_1 pK}{pK+q} - \frac{\alpha_2qK}{qK + p} - \frac{\alpha_1 q}{pK+q} + \frac{\alpha_2 p }{qK + p}\\
&= 2 - \frac{2\alpha_2qK}{qK + p} - \frac{2\alpha_1 q}{pK+q}.\\
\end{align*}
Then the costs of achieving $\frac{1}{2}||C_1||_t = d_{TV}(\pi_t,\pi'_t) \le \varepsilon$ become
\begin{align*}
{\rm Cost}(\alpha_1,\alpha_2)=\frac{\log(1/\varepsilon) (\alpha_1Kn+\alpha_2n)}{\log\left(\left(1 - \frac{\alpha_2qK}{qK + p} - \frac{\alpha_1 q}{pK+q}\right)^{-1}\right)}= \frac{n \log(1/\varepsilon) (\alpha_1K+\alpha_2)}{\frac{\alpha_2qK}{qK + p}\ (1+o(1))}.
\end{align*}
First, note that the main term above is increasing in $\alpha_1$. If $\alpha_1>0$, which is suboptimal, then the minimum is achieved with  $\alpha_2=1$. However, interestingly, if $\alpha_1=0$ then $\alpha_2$ cancels. Indeed, as compared to the policy when $G_t=B_2$, we reduce $||C_t||_1$ by the factor $\left(1-\frac{\alpha_2qK}{qK+p}\right)$ instead of $\left(1-\frac{qK}{qK+p}\right) = \frac{p}{qK+p}$, but we also need only the fraction $\alpha_2$ for computational efforts. For large $n$, it turns out that the loss and the gain roughly compensate each other.

The question arises, how many nodes of $B_2$ should receive green light at the same time? Recall that after step~0, nodes of $B_1$ have negative cash and nodes in $B_2$ have positive cash. If we give green light only to nodes in $B_2$, the sign of the cash will not change for any of the nodes. Furthermore, whenever a node in $B_2$ receives a green light, it sends fraction $\frac{qK}{qK+p}$ of its cash to $B_1$, so this cash disappears from the system, but fraction $\frac{p}{qK+p}$ is distributed over other nodes in $B_2$ and stays in the system. This part can be depleted from the system when green light is given to the other nodes in $B_2$. Then the obvious greedy scheduling to deplete the cash from the system as fast is possible, is  to always give green light to a node in $B_2$ with maximal cash. Due to the symmetry of the system, this will result in giving green light to nodes of $B_2$ in a cyclic order. Without loss of generality assume that this cycle goes through the nodes from $Kn+1$ to $(K+1)n$,  and $C_{t,Kn+1} \ge C_{t,Kn+1} \ge \cdots \ge C_{t,(K+1)n}$, where $t$ is the starting time of the cycle. Then it is easy to compute that the node $Kn+\delta n$, where $\delta\in (0,1)$, will deplete from the system the amount of cash equal to
\[\frac{qK}{qK+p}\,C_{t,Kn+\delta n} + \frac{qK}{qK+p}\, \frac{p}{qK+p} \frac{1}{n} \left(\sum_{i=Kn+1}^{Kn+\delta n -1} C_{t,i}\right) (1+o(1)) \ge
\]
\[\frac{qK}{qK+p}\,C_{t,Kn+\delta n} \left(1+\frac{\delta p}{qK+p}\, (1+o(1))\right).\]
Hence, at the end of the cycle the total amount of cash remaining in the system will be smaller than $\frac{p}{qK+p}\,||C_t||_1$, so compared to $G_t=B_2$ more cash will be depleted at the same computational costs.

This suggests that an effective policy is to give green light to the nodes with highest amount of cash. Our numerical results show that such policies indeed perform well, however, they do not need to be optimal. In the next section we will show that finding optimal green light scheduling is a challenging task even on a very restricted set of policies in a quite simple mean-field SBM with three blocks.

\section{Effectiveness of cash-dependent green-light scheduling}
\label{sec:three_blocks}

The schedule that always gives green light to the smallest block, as discussed in Section~\ref{sec:two_blocks} for the two-block mean field SBM, could be considered as a particular case
of Markovian, cash-independent, schaduling strategies: we select the smallest block with probability one.

In this section we will show that in general one needs to consider truly cash-dependent scheduling strategies. As an  example, we will use the mean-field SBM with three blocks. As before, we assume that the density of the intra-block links is $p$, and the density of the inter-block links is $q<p$, so the transition matrix is as follows:
\begin{equation}
\label{eq:P3}
P = \left[\begin{array}{ccc}
\frac{p}{N_1(p-q)+N} J_{N_1,N_1} & \frac{q}{N_1(p-q)+N} J_{N_1,N_2} & \frac{q}{N_1(p-q)+N} J_{N_1,N_3}\\
\frac{q}{N_2(p-q)+N} J_{N_2,N_1} & \frac{p}{N_2(p-q)+N} J_{N_2,N_2} & \frac{q}{N_2(p-q)+N} J_{N_2,N_3}\\
\frac{q}{N_3(p-q)+N} J_{N_3,N_1} & \frac{q}{N_3(p-q)+N} J_{N_3,N_2} & \frac{p}{N_3(p-q)+N} J_{N_3,N_3}
\end{array}\right].
\end{equation}

In the remainder of this section we will analyze the class of scheduling strategies when all nodes of one block move their cash only together (we will also say that the cash is moved by blocks). In Subsection~\ref{ssec:3block_strategies} we formally describe this class of scheduling strategies and analyze the scenario when, similarly to Section~\ref{sec:two_blocks}, only one block moves its cash. Interestingly, we will see that such schedule does not guarantee the convergence of the RLGL algorithm. Next,
in Subsection~\ref{ssec:DP} we formulate and solve an optimization problem that obtains the optimal green-light scheduling. We demonstrate that the optimal scheduling is cash-dependent and provide numerical results that display its effectiveness.

\subsection{Green-light scheduling strategies when cash is moved by blocks}
\label{ssec:3block_strategies}

If we move the cash by blocks, the symmetry is preserved, and we can describe the state of the system by the cash row-vector
\[
C_t = [c_{t,1}\onevec_{N_1}^T \ c_{t,2}\onevec_{N_2}^T \ c_{t,3}\onevec_{N_3}^T],
\]
with
\[
c_{t,i}=\pi_{t,i}-\pi'_{t,i}, \quad i=1,2,3,
\]
and the total cash in the system is zero:
\begin{equation}
\label{eq:3blocksCashBalance}
N_1 c_{t,1} + N_2 c_{t,2} + N_3 c_{t,3} = 0.
\end{equation}
Since the cash is moved by blocks, we may give green light simultaneously to one, two or three blocks. The corresponding changes in the cash are as follows:

\begin{description}
\item[(i)] If $G_t=\{i\}\subset \{1,2,3\}$, then

\begin{align}
\label{eq:c1ii}
c_{t+1,i} &= c_{t,i} \frac{N_ip}{N_i(p-q)+Nq},&\\
\label{eq:c1ij}
c_{t+1,j} &= c_{t,j} + c_{t,i} \frac{N_iq }{N_i(p-q)+Nq},& j\in \{1,2,3\}\backslash \{i\}.
\end{align}

\item[(ii)] If $G_t=\{i,j\}\subset \{1,2,3\}$, then
\begin{align}
\label{eq:c2ii}
c_{t+1,k} &= c_{t,k} \frac{N_kp}{N_k(p-q)+Nq} + c_{t,l} \frac{N_lq}{N_l(p-q)+Nq},& \{k,l\}=\{i,j\},\\
\label{eq:c2ij}
c_{t+1,k} &= c_{t,k} + c_{t,i}\frac{N_iq}{N_i(p-q)+Nq} + c_{t,j} \frac{N_jq}{N_j(p-q)+Nq},& \{k\}=\{1,2,3\}\backslash \{i,j\}.
\end{align}

\item[(iii)] If $G_t=\{1,2,3\}$, then
\begin{align}
\label{eq:c3}
c_{t+1,i} &= c_{t,i} \frac{N_i p}{N_i(p-q)+Nq} + c_{t,j} \frac{N_jq}{N_j(p-q)+Nq} + c_{t,k} \frac{N_kq}{N_k(p-q)+Nq},& \{i,j,k\}=\{1,2,3\}.
\end{align}
\end{description}


We now would like to obtain insight on how to chose the sequence of actions ${\bf G}=(G_1,G_2,\ldots)$ so that the cost of the RLGL algorithm is minimized. Interestingly, if at each step $t$ we move the cash from the same block, then, in general, and in contrast to the situation in the two-block SBM,
the algorithm will not converge. This is stated formally in the next theorem.

\begin{theorem}
\label{th:divergence3}
Let $P$ be given by \eqref{eq:P3}. Fix $i=1,2,3$  and let $G_t=\{i\}$ for all $t\ge 1$. Denote by $t_0$ the time step when for the first time the cash in one of the blocks becomes zero or changes sign:
\[t_0=\min\{t: \sgn(c_{t,j}) = - \sgn(c_{t-1,j}) \mbox{ or $c_{t,j} =0$ for some $j=1,2,3$}\}.\]
Then the RLGL algorithm converges only if $\sgn(c_{1,i})=-\sgn(c_{1,j})$ for each $j\in \{1,2,3\}\backslash \{i\}$ and $t_0=\infty$. In all other cases
\[\liminf_{t\to \infty} ||C_t||_1>0,\]
so the RLGL algorithm does not converge.
\end{theorem}

\begin{proof} Without loss of generality, take $i=1$ and assume that at time $t=1$ the cash in block~1 is positive, $c_{1,1}>0$. Since we give green light only to block~1, it follows from \eqref{eq:c1ii} that its cash will remain positive and decrease exponentially in $t$:
\[c_{t,1}=c_{1,1}\left(\frac{N_ip}{N_i(p-q)+Nq}\right)^{t-1}>0, \quad t \ge 1.\]
In the limit, the total amount of cash transferred by block~1 is $c_{1,1}N_1$, out of which block~$j=2,3$ receives

\begin{equation}
\label{eq:received}
\sum_{t=1}^\infty\mbox{[cash received by block~$j$ at time $t$]} = \frac{N_j}{N_2+N_3}\,c_{1,1}N_1,\quad j=2,3.
\end{equation}
We will consider three possible cases.

{\bf Case~1.} Suppose that $c_{1,2},c_{1,3}<0$ and $t_0=\infty$. (This happens, for example,  when $N_2=N_3$ so the amounts of cash in blocks~2 and 3 are identical: $c_{t,2}=c_{t,3}= - c_{t,1}/2<0$, $t\ge 1$.) Since $t_0=\infty$ implies that $c_{t,j}<0$, $j=2,3$, for all $t\ge 1$, and the total amount of cash in the system is zero, we have that
\[||C_t||_1 = 2|c_{t,1}|N_1\to 0\quad \mbox{as $t\to\infty$},\]
so the RLGL algorithm converges.

{\bf Case 2.} Suppose that $c_{1,2},c_{1,3}<0$ and $t_0$ is finite. This happens when for some block $j=2,3$, the absolute value of cash at time $t=1$, $|c_{1,j}|N_j$, is smaller than the right-hand side of \eqref{eq:received}. Then
\[\liminf_{t\to\infty} ||C_t||_1 \ge \lim_{t\to\infty} N_j|c_{t,j}| =
\frac{N_j}{N_2+N_3}\,c_{1,1}N_1 - c_{1,j}N_j>0,\]
so the RLGL algorithm does not converge.

{\bf Case 3.} Suppose that $c_{1,j}>0$ for some $j=2,3$. Since block~$j$ receives only positive cash, its cash can only increase, so we have
\[\liminf_{t\to\infty} ||C_t||_1 \ge \lim_{t\to\infty} N_j c_{t,j}  =
\frac{N_j}{N_2+N_3}\,c_{1,1}N_1 + c_{1,j}N_j>0,\]
and the RLGL algorithm does not converge.
\end{proof}

\begin{remark}
Note that the RLGL algorithm converges only in Case~1 in the proof of Theorem~\ref{th:divergence3}. In this case it must hold that $\frac{N_j}{N_2+N_3}\,c_{1,1}N_1 = c_{1,j}N_j$, $j=2,3$, so this case is quite exceptional and requires some balancing of the parameters, e.g., $N_2=N_3$.
\end{remark}

\begin{remark} In the mean field SBM all nodes in the same block are identical, and therefore they can be lumped into one state of an aggregated Markov chain with states $i=1,2,3$. Then the results in Section~\ref{sec:two_blocks} and Theorem~\ref{th:divergence3} show that for Markov chains with more than two states, if green light is given only to one state $i$, then, in general, the RLGL algorithm does not converge. Again, the exception corresponds to Case~1 in the proof of Theorem~\ref{th:divergence3}.
\end{remark}

From Theorem~\ref{th:divergence3} we conclude that routing exclusively the cash of a single block does not result in convergence of the RLGL algorithm (except in some special cases). However, $G_t=\{i\}$ can be a good choice for some time period. The above analysis already indicates that
the structure of an optimal scheduling for green light is non-trivial.

In order to find optimal strategies for green-light scheduling, we will next formulate the problem of optimal block schedule as an optimal control problem.

\subsection{The optimal block schedule}
\label{ssec:DP}

The theory of Markov Decision Processes (MDP) \cite{Puterman,Hernandez2012} provides
an appropriate framework for controlled Markov chains with total cost criterion.
The main components of an MDP model are: system states, available actions and one-step
costs.

As the first obvious choice for the state variables of the MDP model, one could take $c_{t,1}$ and $c_{t,2}$ ($c_{t,3}$ can be immediately restored using the relation (\ref{eq:3blocksCashBalance})).

In our case the MDP action space consists of 7 actions, i.e., $A=\{a_1,...,a_7\}$, with actions corresponding to the sets of blocks receiving green light. Specifically, let
$$
a_1=\{1\}, \ a_2=\{2\}, \ a_3=\{3\}, \ a_4=\{1,2\}, \ a_5=\{2,3\}, \ a_6=\{1,3\}, \ a_7=\{1,2,3\}.
$$
For instance, if we choose the action $a_5$, we give green light simultaneously to blocks 2 and 3.

If $||C_t||_1 > \varepsilon$ for the chosen precision $\varepsilon>0$, the one-step cost
$\kappa(a)$ of action $a$ is equal to the average number of operations in the SBM graph
with the given choice of the routing blocks:
\begin{align*}
\kappa(a_1) &= N_1(pN_1+qN_2+qN_3),\\
\kappa(a_2) &= N_2(qN_1+pN_2+qN_3),\\
\kappa(a_3) &= N_3(qN_1+qN_2+pN_3),\\
\kappa(a_4) &= \kappa(a_1)+\kappa(a_2),\\
\kappa(a_5) &= \kappa(a_2)+\kappa(a_3),\\
\kappa(a_6) &= \kappa(a_1)+\kappa(a_3),\\
\kappa(a_7) &= \kappa(a_1)+\kappa(a_2)+\kappa(a_3).
\end{align*}
If $||C_t||_1 \le \varepsilon$ then the desired precision is achieved, and the cost of any action is zero.
We aim to find a green-light scheduling strategy that optimizes the total expected cost:
\begin{equation}
\label{eq:TotalCost}
V(c)=E[ \ \sum_{t=0}^{\infty} \kappa(G_t) \ | \ c_0=c \ ] \to \min,
\quad G_t \in A, \ t\ge 0,
\end{equation}
when we start the process from state $c_0=c$.
The above criterion means that we would like to reach the precision $\varepsilon$ spending
the least number of operations possible.
It is known \cite{Hernandez2012} that we can limit the search to the stationary deterministic strategies.
Such an optimal strategy is described by the following dynamic programming equation:
\begin{equation}
\label{eq:dynprog}
V(c) = \min_a \left[ \kappa(a) + V(f_a(c)) \right],
\end{equation}
where $f_a(c)$ is the next state provided that the current state is $c=(c_1,c_2)$ and action $a$ is applied.

Let us first show that we can limit the search of optimal policy to the half plane $c_1 \ge 0$.
This follows from the next theorem.

\begin{theorem}
\label{th:Vproperty}
The value function $V(c)$ and the sets of optimal actions $A(c)$ have the following property:
$$
V(-c)=V(c), \quad A(-c)=A(c).
$$
\end{theorem}

\begin{proof} As often the case in the MDP theory, the structural properties can be proved
with the help of the value iteration algorithm. Let now $n$ denote the running index of the
value iteration algorithm. We emphasize that the index $n$ does not correspond to the time.
From \cite{Puterman,Hernandez2012}, we know that the value iterations
$$
V_{n+1}(c) = \min_a \left[ \kappa(a) + V_n(f_a(c)) \right],
$$
$$
A_{n+1}(c)=\arg \min_a \left[ \kappa(a) + V_n(f_a(c)) \right],
$$
with $V_0(c) = 0$ for all $c$, and $V_n(c)=0$ for all $n$ and $c$ such that $\Vert c \Vert_1<\varepsilon$,
converge to the value function $V(c)$ and optimal control $A(c)$.
Note that $A_{n}(c)$ is only defined
for $\Vert c \Vert_1\ge\varepsilon$ and the convergence of $V_n(c)$ is monotone in $n$.

We show that $V_n(-c)=V_n(c)$ and $A_n(-c)=A_n(c)$ for any $n$ by induction.
It is evidently true for $n=0$. Let us assume it is true for $n=k$.
If $\Vert c \Vert_1<\varepsilon$ then by the boundary condition $V_{k+1}(c)=V_{k+1}(-c)=0$.
For $\Vert c \Vert_1\ge\varepsilon$, we can write:
\begin{align*}
V_{k+1}(c) &= \min_a \left[ \kappa(a) + V_k(f_a(c)) \right]\\
 &= \min_a \left[ \kappa(a) + V_k(-f_a(c)) \right] &&\text{\;(true by assumption in step \:}n=k)\\
 &= \min_a \left[ \kappa(a) + V_k(f_a(-c)) \right]&&\text{\;(since\:}f_a \text{\; is linear in $c$})\\
 &= V_{k+1}(-c). \\
\end{align*}
Similarly,
\begin{align*}
A_{k+1}(c)&=\arg \min_a \left[ \kappa(a) + V_k(f_a(c)) \right],\\
&=\arg \min_a \left[ \kappa(a) + V_k(f_a(-c)) \right],\\
&=A_{k+1}(-c).
\end{align*}
Thus, by induction, the statement of the theorem follows.
\end{proof}

\bigskip

The above theorem implies that we can limit ourselves to the part of the state space with $c_1 \ge 0$.
In particular, if the system moves out of the half plane $c_1 \ge 0$, we can just take an optimal
action corresponding to $-c$. Due to Theorem~\ref{th:Vproperty}, $V(-c)=V(c)$, and the total optimal cost starting from $c$ will be the same.

Despite the above simplification of the state space, we still face a serious computational problem
because the state space is very ``crowded'' around the point (0,0). From the analysis performed
at the beginning of this section and in Subsection~\ref{sec:dobrushin} we can expect that the convergence
will be at least exponentially fast.
Thus, we suggest to choose as the first component of the state vector the logarithmic transformation
of the total absolute cash, i.e., $\log_{10}(||c||_1/\varepsilon)$. The choice of the second component
of the state space becomes more tricky. We suggest the following choice:
$$
z=(z_1,z_2)=\left(\log_{10} (||c||_1/\varepsilon), y_2-y_3\right),
$$
where  $y_i=2 N_ic_i / ||c||_1$ and assuming $c_1\ge 0$. Let us show that $z$ describes the considered half space. Obviously, $z_1\in [0,\log_{10} (||c_0||_1/\varepsilon)]$ and we can immediately retrieve $||c||_1$ from $z_1$.
Let us next show that there is a one-to-one correspondence between $z_2$ and $(y_2,y_3)$, given
$y_1+y_2+y_3=0$ and $c_1 \ge 0$. We shall use the following set of equations:
\begin{align*}
y_1+y_2+y_3&=0,\\
|y_1|+|y_2|+|y_3|&=2,\\
y_1^+ +y_2^+ +y_3^+&=1,\\
y_1^- +y_2^- +y_3^-&=-1,
\end{align*}
where $y^+=\max(y,0)$and $y^-=\min(y,0)$. Clearly, the values of $y_2$ and $y_3$ uniquely define $z_2=y_2-y_3$.
Let us also show that $z_2$ covers continuously the interval $[-2,+2]$.
We need to consider three cases depending on the possible value of $y_3$:
\begin{itemize}
\item[i)] $y_3=-1$: From the equations $y_1^+ + y_2^+=1$ and $y_1\ge 0$, we conclude that $y_2 \le 1$,
and from the equation $y_2^-=0$, we conclude that $y_2 \ge 0$. Thus, $y_2\in [0,1]$ and hence
\begin{align*}
z_2=y_2-y_3=y_2+1\in [1,2].
\end{align*}
\item[ii)] $y_3 \in (-1,0]$: From the equation $y_2^-+y_3=-1$, it follows that $y_2=-1-y_3 \in [-1,0)$
and hence
\begin{align*}
z_2=y_2-y_3=-1-2y_3\in [-1,1).
\end{align*}
\item[iii)] $y_3 \in (0,1]$: From the equation $y_2^-=-1$, we immediately have $y_2=-1$. Then, from
the equation $y_1^+ +y_3^+=1$ we conclude that $y_3\le 1$ and thus
\begin{align*}
z_2=y_2-y_3=-1-y_3\in [-2,-1).
\end{align*}
\end{itemize}
Using the derived equations for $z_2$ for the three cases considered above, we can express $y_2$
and $y_3$ as functions of $z_2$. Namely,
\begin{itemize}
 \item[i)] $z_2\in [1,2]$: $y_2=z_2-1$ and $y_3=-1$;
 \item[ii)] $z_2\in [-1,1)$: $y_2=(z_2-1)/2$ and $y_3=-(z_2+1)/2$;
 \item[iii)] $z_2\in [-2,-1)$: $y_2=-1$ and $y_3=-z_2-1$;
\end{itemize}
and thus we conclude that there is a one-to-one correspondence between $z_2$ and $(y_2,y_3)$.

We can solve dynamic programming equation \eqref{eq:dynprog} by discretization of the state space.
One more important advantage of this modified state space is that it allows us to efficiently
organize the optimization process due to the fact that $z_1$, being the logarithm of $||c||_1$,
is non-increasing for all actions. Indeed, we can use the backwards induction, since
$V(z_1,z_2)$ depends only on $V(z'_1,z'_2)$ such that $z_1 > z_1'$ (every action reduces the total cash).
Thus, the value function $V(z_1,z_2)$ can be calculated in one pass starting from $V(0,z_2) \equiv 0$.
We observed that the discretization of $z_2$ does not need to be too fine. This allows us to use much
finer grid along the $z_1$-axis.


Let us apply this approach to determine the optimal block policy
for a mean-field SBM with representative parameters. Specifically, we choose:
$N_1=50$, $N_2=20$, $N_3=10$, $p=0.1$, $q=0.01$, $\varepsilon=10^{-14}$.
We have discretized the [0,14] interval of the $z_1$-axis with 14000 points and the [-2,+2] interval
of the $z_2$-axis with 321 points. Using the backwards induction, we have obtained
the optimal actions depicted in Figure~\ref{fig:OptActions}.

\begin{figure}[ht!]
     \centering
     \begin{subfigure}[b]{0.45\textwidth}
	   \includegraphics[width=0.98\textwidth]{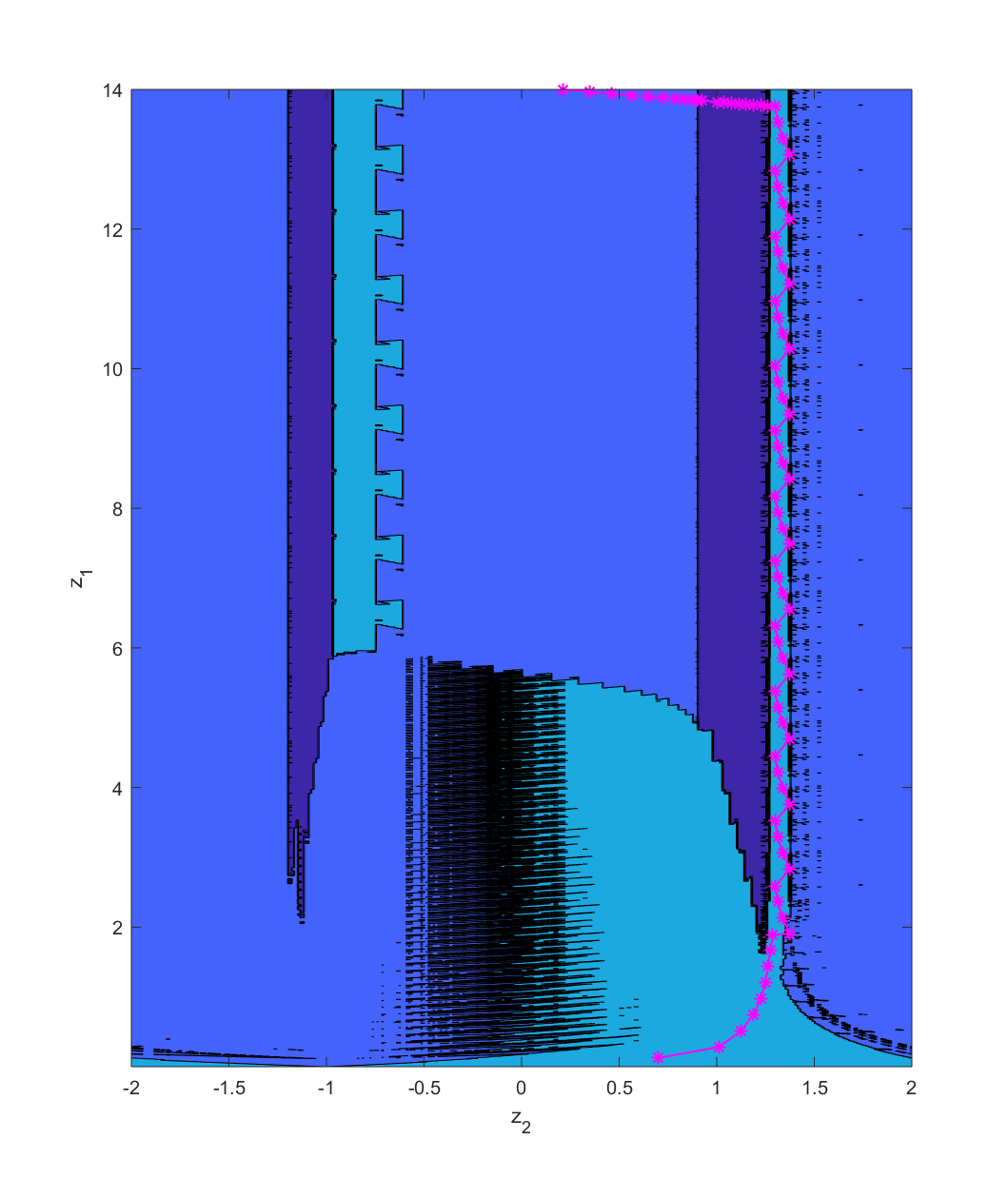}
       \caption{Actions in the complete state space.}\label{fig:OptActionsComp}
     \end{subfigure}
     \hfill
     \begin{minipage}[b]{0.45\textwidth}
       \begin{subfigure}[b]{\linewidth}
        \centering
	    \includegraphics[width=0.7\textwidth]{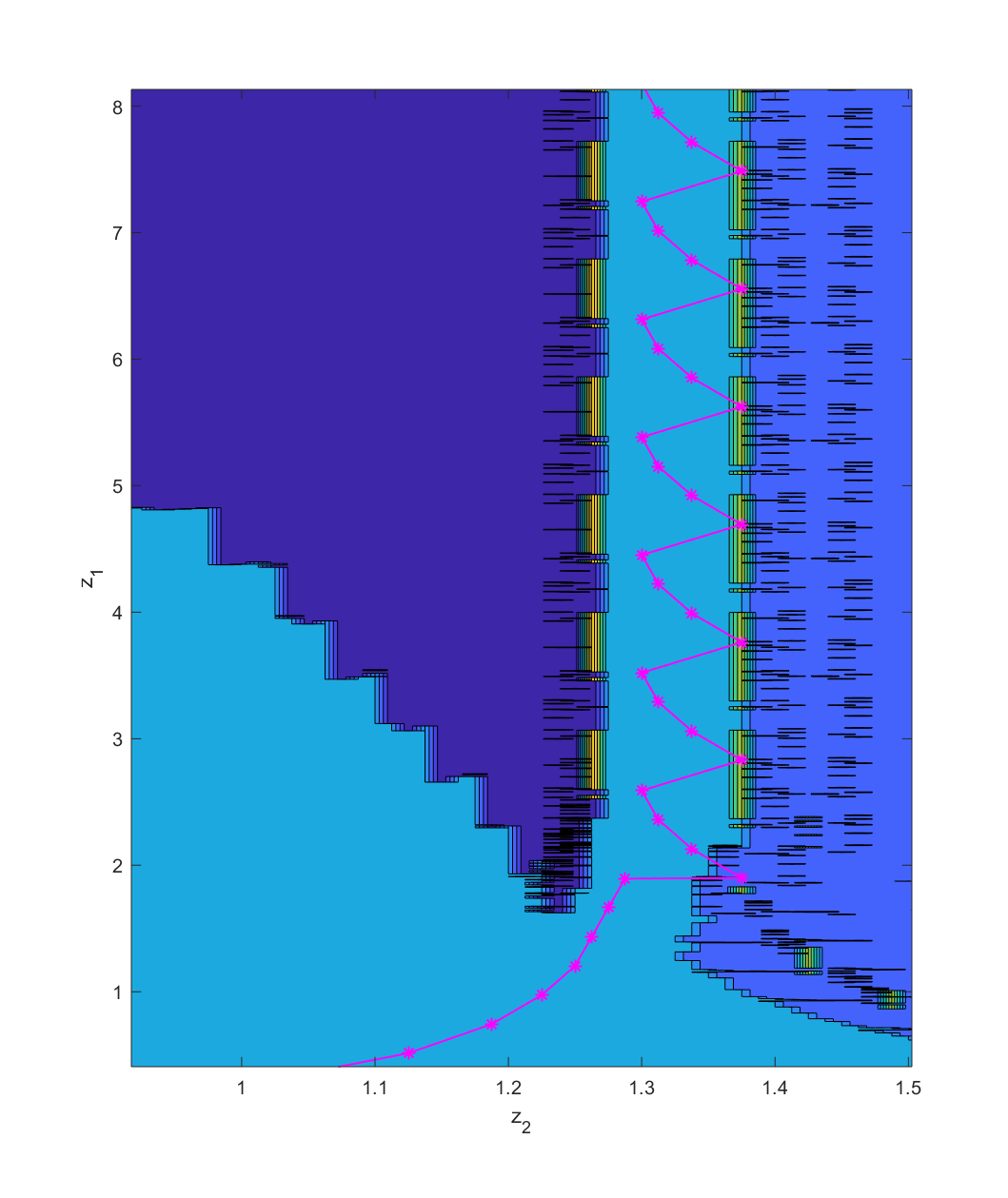}
        \caption{Zoom on the strategy along the turnpike.}\label{fig:OptActionsZoom}
       \end{subfigure}\\[\baselineskip]
       \begin{subfigure}[b]{\linewidth}
        \centering
	    \includegraphics[width=0.8\textwidth]{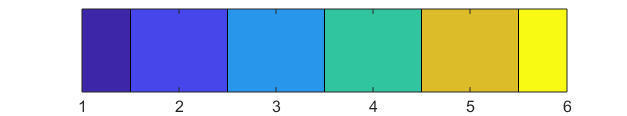}
        \caption{Colours denoting actions.}\label{fig:ActionsScale}
       \end{subfigure}
     \end{minipage}
     \caption{[Best viewed in colour] The optimal green-light scheduling in the mean-field SBM with three blocks.
     The magenta line corresponds to the optimal RLGL trajectory starting from one PI initialization.} \label{fig:OptActions}
\end{figure}

\begin{figure}[h!]
     \centering
     \begin{subfigure}[b]{0.45\textwidth}
       \centering
	   \includegraphics[width=0.99\textwidth]{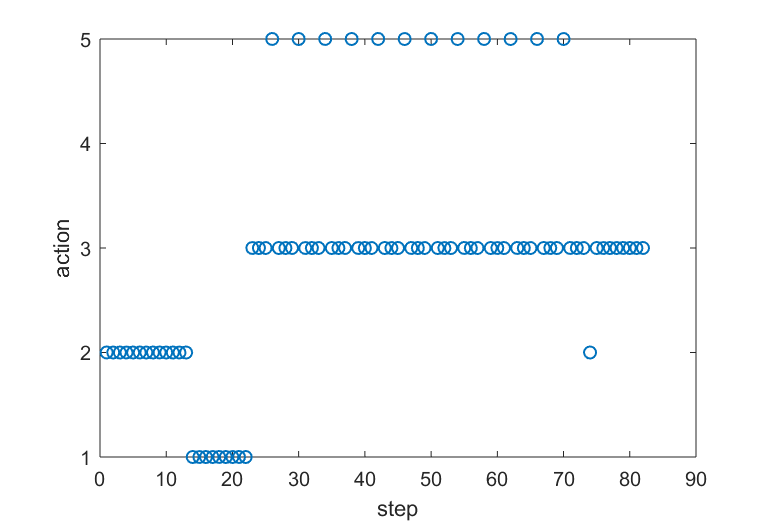}
       \caption{Actions taken by the optimal schedule.}\label{fig:ActionsTaken}
     \end{subfigure}
     \hfill
     \begin{subfigure}[b]{0.45\textwidth}
       \centering
	   \includegraphics[width=0.9\textwidth]{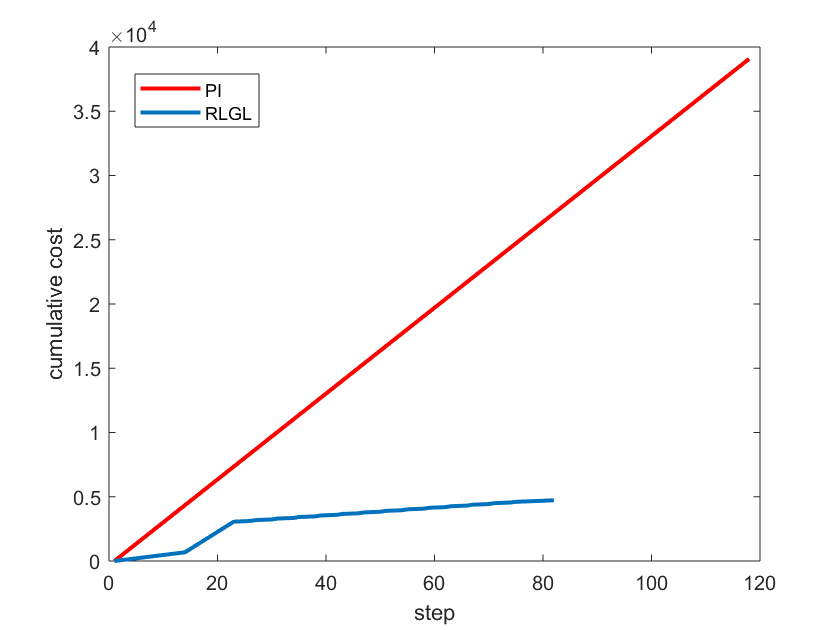}
       \caption{Comparison of cumulative costs: RLGL vs PI.}\label{fig:CostsRLGLvsPI}
     \end{subfigure}
     \caption{Mean-field SBM with three blocks.}\label{fig:MFSBM3blocks}
\end{figure}

%
%

The vast majority of optimal actions in the action space consists of actions $a_2=\{2\}$ and $a_3=\{3\}$,
which makes sense as these actions correspond to routing the cash from the two smallest
blocks.

Recall that the basic version of RLGL creates the initial cash distribution with one
step of power iteration, which results in
$$
c_{0,i} = \frac{\zeta}{N} \sum_{j=1}^{3} N_j \frac{N_j-N_i}{(N_i\zeta+N)(N_j\zeta+N)}
$$
(according to our convention, $N_1 \ge N_2 \ge N_3$, $N_1+N_2+N_3=N$). Then we can see from
the above expression that the largest block (Block~1) initially has negative
cash and the smallest block (Block~3) has positive cash. The cash of Block~2 depends on the system parameters.

We have run RLGL with the optimal green-light schedule obtained from the dynamic programming
and have superposed the obtained optimal trajectory on the action plan, see the magenta line
in Figures~\ref{fig:OptActionsComp},\ref{fig:OptActionsZoom}. It is very interesting that we can
observe a manifestation of the turnpike principle \cite{mckenzie1976turnpike}.
This principle could be intuitively
explained by a highway (turnpike), which a car first needs to join and then to leave close to
the target. In our case, the turnpike consists of a periodic pattern of actions $(a_3,a_3,a_3,a_5)$,
see Figures~\ref{fig:OptActionsComp},\ref{fig:OptActionsZoom}. Note that $a_5$ gives green light
simultaneously to Blocks 2 and 3, the two smallest blocks.
In Figure~\ref{fig:ActionsTaken} we provide a complete sequence of optimal actions along
the RLGL trajectory starting from PI initialization.
Most of the actions taken are either $a_2=\{2\}$ or $a_3=\{3\}$, which corresponds to least effort.

For this optimal sequence of actions,
we compare in Figure~\ref{fig:CostsRLGLvsPI} the running costs of RLGL and PI.
It is quite remarkable that
not only the total cost of RLGL is about 10 times smaller than the total cost of PI but also
this cost reduction takes place consistently during nearly the whole run.
Moreover, RLGL achieves the goal with fewer updates than the number of the PI iterations.
This makes a contrast with the two blocks example, where PI converges faster but at a much
higher cost. In fact, the action $a_7=\{1,2,3\}$, which corresponds to PI,
is not present in the optimal policy.

%

Of course, we are conscious of the fact that in practice it is unreasonable to solve a complicated dynamic programming problem to elaborate an optimal schedule of updates. However, we can draw several conclusions
useful for practice from this exercise: firstly, by scheduling the routing of cash in a right way one can obtain significant gains
with respect to the standard method of power iterations. Secondly, one should schedule updates taking
into account the current cash value. And thirdly, one should try to schedule smaller
subsets of states (possibly with a large value of cash aiming to annihilate as much cash as possible).


\section{Numerical results}
\label{sec:num}

As we saw in the previous section, for a general Markov chain, the optimal cash routing requires the solution of a dynamic programming problem with an exponential number of actions. Clearly this is not feasible for any reasonable size of transition matrix and calls for the elaboration of  heuristics. In fact in Section \ref{sec:exp_convergence} we have shown an exponential convergence of the following two cash-independent heuristics for green light scheduling.

\bigskip
\textbf{Cash-independent heuristics:}
\begin{description}
\item[RLGL\_RR,] $G_t = \{i \mid  i \equiv t \pmod N\}$, nodes route their cash in a cyclic or round robin manner as with Gauss-Seidel. As numerical examples will demonstrate, RLGL\_RR converges typically with a rate twice faster than PI. Intuitively as in the case of Gauss-Seidel (GS), RLGL\_RR uses the most recently updated information. This possibly explains the superior performance in comparison with PI. Furthermore, RLGL\_RR can be implemented in a distributed way by token passing.
\item[RLGL\_Rand,] $G_t = \{i\}$ with probability $1/N$, nodes are chosen  randomly. The performance of this heuristic is slightly worse than that of PI, as suggested by the discussion at the end
    of Subsection~\ref{ssec:geometric}. One significant advantage of this heuristic is its easy asynchronous implementation.
\end{description}

In the previous
Sections~\ref{sec:two_blocks}~and~\ref{sec:three_blocks} we observed that giving green light to the nodes with a large absolute value of cash can be more effective. Therefore, we also propose and test a series of heuristics with cash-dependent scheduling, in particular, heuristics prioritizing green-light scheduling to nodes with large absolute value of cash.

\bigskip
\textbf{Cash-dependent heuristics:}

\begin{description}
\item[RLGL\_Greedy,] $G_t = \argmin_{i}{\Vert C_{t+1} \Vert}_1$, the green light is given to the node which minimizes the next total absolute cash at the next step. {Somewhat surprisingly, it has not
    performed well on several examples.}
\item[RLGL\_MaxC,] $G_t = \{i \mid C_{t,i} = \max_j |C_{t,j}|\} $, the green light is given to a node with maximum absolute cash. This heuristic performs  well in general as we will see in the ensuing numerical examples. However, this heuristic raises two concerns: can it scale to large systems and can it be distributed?
\item[RLGL\_PC,] $G_t = \{i\}$ with probability proportional to absolute value of cash
$|C_{t,i}|$. This heuristic looks like a good compromise between RLGL\_Rand and RLGL\_MaxC. It can be asynchronously distributed with a Poisson clock where node $i$ is updated after continuous time distributed according to the probability distribution
$1-e^{-|C_i|t}$.
In contrast to RLGL\_Rand, RLGL\_PC gives advantage to nodes with a large absolute amount of cash.
However, since the total value of cash decreases in time the intervals between updates will grow.
We elaborate on the mitigation of this issue a bit later in this section.
\item[RLGL\_Theta,] $G_t = \{i \mid  i \equiv t \pmod N\ \text{ and } |C_{t,i}| \ge \theta_t(r) \}$, where $\theta_t(r) = \left(\sum_j \frac{|C_{t',j}|^r}{N} \right)^\frac{1}{r}$, $r\ge 1$,
    $t'= \lfloor t/N \rfloor$, i.e.,
    nodes route their cash if their cash exceeds a threshold, which is updated periodically.
    This heuristic becomes more selective as we increase $r$ and hence the value of the threshold, approaching RLGL\_MaxC.
To make this heuristics more scalable, the threshold can be  updated less frequently and not necessarily periodically. As will be shown later in this section, a big advantage of RLGL\_Theta is its ability to overcome the clustering structure in networks.
\end{description}

We will compare the above RLGL heuristics with the following \textbf{reference methods}:
\begin{description}
\item[PI] The power iteration method: $\hat \pi_{t+1}=\hat \pi_tP$, $t\ge 0$.
\item[GS] The Gauss-Seidel method: $\hat \pi_{t+1,j}=\left(\sum_{i<j} \hat\pi_{t+1,i}p_{i,j}+\sum_{i>j} \hat\pi_{t,i}p_{i,j}\right)/(1-p_{i,i})$, $t\ge 0$.
\item[GMRES] The GMRES method belongs to a family of methods for solving linear systems $Ax=b$  based on Krylov subspace: $K_{M+1}=\mbox{span}\{r_0,r_0A,\dots,r_0A^M\}$, where $M+1$ is the dimension of the Krylov subspace considered and $r_0=b-Ax_0$.  GMRES approximates the exact solution of $Ax=b$ by the vector $x_{M+1} \in K_{M+1}$ that minimizes the Euclidean norm of the residual $Ax_{M+1}-b$. In the case of Markov chains we take:
$$
A=P^T-I,
\quad
b={\bf 0},
$$
and use as initial guess a vector $x_0\ne {\bf 0}$ \cite{Stewart1992GMRES_Markov}.

\item[GSo-PR] The Gauss-Southwell method, which is only applied to PageRank, but it is an important application. See its detailed description in Algorithm~\ref{alg:GSo}.
For comparison, we will apply some scheduling heuristics for choosing $G_t$ as in RLGL. For instance, GSo-PR\_RR stands for the version of Gauss-Southwell method scheduling green light in round robin fashion.
\end{description}

Since our main aim is the computation of stationary distributions and/or PageRank for large Markov chains, web graphs naturally provide challenging examples. Besides their size, web graphs present even more challenges: heterogeneity in terms of node degrees, clustering structure, absorbing strongly connected components.
In addition to the web graphs we also tested two synthetic graphs: a stochastic block model graph and a two wheels graph. The stochastic block model presents clustering structure and the two wheels graph (see Figure~\ref{fig:twowheels}), despite being a simple graph, in addition to its clustering structure, presents significant degree heterogeneity. We have also compared
numerically the optimal solution with various heuristics on the mean-field three-block SBM.
We summarize the graph data in Table~\ref{tab:graphs}. In Tables~\ref{tab:gmres} and \ref{tab:gmres2} we compare the average  processing times over $10$ runs of RLGL and GMRES for the residual to be lower than a value $\varepsilon$, i.e.,
until the condition $||\hat \pi_{t}-\hat \pi_{t}P||_1  < \varepsilon$ is satisfied.
When comparing with the other methods in Figures~\ref{fig:harvard500} to \ref{fig:3blocks-MF-opt}, we plot for each method the \textit{error}, $||\hat \pi_{t}-\hat \pi ||_1$,  as a function of the number of normalized iterations. The latter quantity is defined as the cumulative number of graph links, used to update the algorithm state, divided by the number of links in the graph. The reference value $\hat{\pi}$, used to measure the error, is an approximation of the stationary distribution $\pi^*$
calculated with the PI method. The PI method is iterated until $||\hat \pi_{t}-\hat \pi_{t+1} ||_1$, where $\hat \pi_{t+1}=\hat \pi_{t}P$, is two orders of magnitude smaller than the minimum error presented in the corresponding figure, at which point we set $\hat \pi = \hat \pi_t$. The quality of this approximation is confirmed by the fact that the estimated error $||\hat \pi_{t}-\hat \pi ||_1$ for the GS method decreases linearly (in log scale), for the range of errors presented, as theoretically predicted for $||\hat \pi_{t}- \pi^* ||_1$.

\begin{table}[!t]
\caption{Summary of graph data.}
\label{tab:graphs}
{\fontsize{8pt}{8pt}\selectfont
\renewcommand{\arraystretch}{1.3}
\begin{center}
\resizebox{\columnwidth}{!}{\begin{tabular}{ l | l | l | l}
\hline Name& Nodes  & 	Edges & Description \\
\hline harvard500 & 500 & 2,636 &  Crawl from www.harvard.edu in 2003, SuiteSparse Matrix Collection\tablefootnote{http://sparse.tamu.edu/}\\
\hline stanford & 281,903 &	2,312,497 &	Web graph of Stanford.edu in 2002,
SNAP Network Collection\tablefootnote{http://snap.stanford.edu/ \label{SNAP}}\\
\hline uk-2007-05-100000 & 100,000& 3,050,615& Extract of the .uk domain,
WebGraph Laboratory Datasets\tablefootnote{http://law.di.unimi.it/ \label{WebGraph}}\\
\hline uk-2007-05-1000000 & 1,000,000 & 41,247,159 & Extract of  the .uk domain, WebGraph Laboratory Datasets\textsuperscript{\ref{WebGraph}}\\
\hline web-google & 916,428 & 5,105,039 & Web graph from Google Contest 2002, SNAP Network Collection \textsuperscript{\ref{SNAP}}\\
\hline two wheels & 12 & 21 & Synthetic graph, see Figure~\ref{fig:twowheels}\\
\hline SBM80 & 80 & 	330 & Synthetic Random SBM \\
\hline SBM800 & 800 & 33394 &	Synthetic Random SBM \\
\hline
\end{tabular}}\end{center}}\end{table}

First of all, we would like to note that RLGL\_PC and RLGL\_Theta nearly always outperform the other methods (see Figures \ref{fig:harvard500PR}-\ref{fig:3blocks-MF-opt}). Figure \ref{fig:uk-2007-05-1000000PR} presents one exception: GSo-PR\_Theta very slightly outperforms RLGL\_Theta on the uk-2007-05-1000000 web graph. However, on the smaller instance uk-2007-05-100000 (see Figure \ref{fig:uk-2007-05-100000PR}), RLGL\_Theta outperforms GSo-PR\_Theta for both values $\theta=1$ and $\theta=2$.

Let us now discuss in some more detail how RLGL compares to baseline general purpose algorithms: PI, GS and GMRES.

On all the tested graphs, as predicted in Subsection~\ref{ssec:geometric}, RLGL\_Rand is slightly
slower than PI but has an advantage of natural asynchronous distributed implementation.

We observe that RLGL\_RR  performs comparably to the GS method.
In principle one could try to optimize the ordering of updates in both GS and RLGL methods by preconditioning
or on the fly. However, in RLGL, as opposite to GS, it is possible to use the value of the current cash as an indication of the next node to update. In addition, the RLGL family of methods, which rely on pushing cash via out-going link, presents another advantage in comparison to GS: typically in the context of web crawling, it is much easier to obtain out-going links than in-coming links.

Finally, we compare RLGL\_Theta ($\theta=1$) to GMRES  with restart. Since the complexity of GMRES iterations increase with the number of iterations, it is not possible to compare RLGL and GMRES in terms of the number of iterations. Therefore, in Tables~\ref{tab:gmres} and \ref{tab:gmres2} we present the CPU runtime  of the two methods\footnote{For GMRES we have used the implementation from GNU Scientific Library (GSL), https://www.gnu.org/software/gsl/.}. In Table~\ref{tab:gmres} we solve the PageRank problem with the damping parameter $0.85$ and in Table \ref{tab:gmres2} we find the stationary distribution for the standard random walk on the largest connected component of the graphs. Clearly the latter problem is much more difficult than the former because of very slow mixing of the random walk on large graphs with clustered structure \cite{stewart1991numerical}. Because of this, on large graphs, all methods failed to converge in reasonable time for accuracy higher than $10^{-5}$.
GMRES requires setting of the dimension of the Krylov subspace until convergence or restart. A default size of Krylov subspace is $M=10$, which indeed in our experiments led to efficient convergence for the PageRank problem. We noticed that for larger graphs $M=5$ works even better. On the other hand larger values of the reset parameter result in slower convergence.    For the problem of computation of the stationary distribution, we noticed that $M=5$ often gives faster convergence but can also lead to unstable behavior. We see that RLGL\_Theta consistently outperformed GMRES on various large graphs on both problems. We emphasize that while using GMRES we need to work with dense rectangular matrices of increasing width, which is likely to significantly slow down GMRES at each subsequent iteration before restarting.

\begin{table}
[!t]
\caption{Comparison of the RLGL\_Theta ($\theta=1)$ and GMRES running times on PageRank (for a final residual at most $\varepsilon=10^{-11}$).}
\label{tab:gmres}
{\fontsize{8pt}{8pt}\selectfont
\renewcommand{\arraystretch}{1.3}
\begin{center}
\resizebox{\columnwidth}{!}{\begin{tabular}{ l | l | l || c || c | c || c | c || c | c  }
\hline Name & Nodes  & 	Edges &   RLGL\_Theta & \multicolumn{6}{c}{GMRES}\\
\hline & & &$\theta=1$ & \multicolumn{2}{c|}{$M=10$}& \multicolumn{2}{c|}{$M=5$}& \multicolumn{2}{c}{$M=20$}\\
\hline & & &sec. & sec. & restarts & sec. & restarts & sec. & restarts\\
\hline  harvard500  & 500 & 2,636 &  \textbf{0.017}  &  0.023  &  5  &  0.024  &  10  &  0.026  &  3 \\
\hline  stanford  & 281,903 &	2,312,497 &  \textbf{1.4}  &  6.5  &  9  &  4.7  &  20  &  18  &  5 \\
\hline  uk-2007-05-100000  & 100,000& 3,050,615 &  \textbf{0.96}  &  1.8  &  5  &  1.5  &  10  &  3.4  &  3 \\
\hline  uk-2007-05-1000000  & 1,000,000 & 41,247,159 &  \textbf{11}  &  24  &  6  &  19  &  12  &  46  &  3 \\
\hline  web-google  & 916,428 & 5,105,039 &  \textbf{6.8}  &  30  &  10  &  21  &  21  &  68  &  5 \\
\hline  SBM80  & 80 & 	330 &  \textbf{0.0099}  &  0.013  &  5  &  0.021  &  10  &  0.013  &  2 \\
\hline  SBM800  & 800 & 33,394 &  0.014  &  0.017  &  2  &  0.019  &  5  &  \textbf{0.013}  &  1 \\
\hline
\end{tabular}}\end{center}}\end{table}

\begin{table}
[!t]
\caption{GMRES running times for stationary distributions (for a final residual at most $\varepsilon=10^{-5}$).}
\label{tab:gmres2}
{\fontsize{8pt}{8pt}\selectfont
\renewcommand{\arraystretch}{1.3}
\begin{center}
\resizebox{\columnwidth}{!}{\begin{tabular}{ l | l | l || c || c | c || c | c || c | c  }
\hline Name & Nodes  & 	Edges &   RLGL\_Theta & \multicolumn{6}{c}{GMRES}\\
\hline & & & $\theta=1$ & \multicolumn{2}{c|}{$M=10$}& \multicolumn{2}{c|}{$M=5$}& \multicolumn{2}{c}{$M=20$}\\
\hline & & & sec. & sec. & restarts & sec. & restarts & sec. & restarts\\
\hline  harvard500 LCC  & 335 & 1963 &  \textbf{0.014}  &  0.023  &  5  &  0.019  &  12  &  0.019  &  2 \\
\hline  stanford LCC  & 150532 & 1576314 &  \textbf{225}  &  359  &  819  &  607  &  4244  &  683  &  312 \\
\hline  uk-2007-05@100000 LCC  & 53856 & 1683102 &  \textbf{0.19}  &  0.43  &  3  &  0.29  &  5  &  0.68  &  1 \\
\hline  uk-2007-05@1000000 LCC  & 480913 & 22057738 &  \textbf{1097}  &  2653  &  1162  &  -  &  -  &  -  &  - \\
\hline  web-google LCC  & 434818 & 3419124 &  \textbf{189}  &  741  &  498  &  553  &  1303  &  615  &  96 \\

\hline
\end{tabular}}\end{center}}\end{table}

The chosen synthetic graphs (SBM and Two Wheels graph) have very strong clustering structure. We have chosen such synthetic graphs in the  purpose of  seeing the effect of the clustered structure on the performance of RLGL. We were glad to observe that for PageRank computation on the synthetic graphs RLGL significantly outperformed GSo\_PR (see Figures \ref{fig:A80PR}, \ref{fig:A800PR}, \ref{fig:twowheelsPR}).
At the same time, RLGL outperforms or performs equally well compared with GSo-PR on the real web graphs.
Moreover, the following three important observations can be made: firstly, the GSo-PR method is applicable only to the PageRank problem, whereas RLGL is applicable to Markov chains with general structure. Secondly, as already discussed
in Section~\ref{ssec:math_RLGL}, the GSo-PR method can be viewed as a particular case of RLGL.
{Thus, it should come with no surprise that a more flexible way of green-light scheduling
in RLGL leads to improved performance.}
Thirdly, the rate of cash depletion in the GSo-PR method is limited by ($1-\alpha$)-fraction of the available cash. Thus, RLGL potentially achieves a faster cash depletion rate, which is confirmed by many of the experiments.

Whenever it can be applied, RLGL\_MaxC generally performs comparably to RLGL\_PC and RLGL\_Theta. (The synthetic two-wheels graph, see Figure \ref{fig:twowheelsM}, makes a noticeable exception where RLGL\_MaxC performs extremely well.)  The clear advantages
of the two latter methods are their scalability and distributivity. When implementing RLGL\_PC and RLGL\_Theta in a distributed way, one should be aware about the decrease of the total absolute value of the cash. To address this issue, one can broadcast from time to time the total amount of cash or each node can estimate its rate of decline of cash (we observe that the exponential convergence phase is achieved very rapidly and this helps estimating the total amount of cash); then each node should renormalize its wakeup probablity or its threshold.

RLGL\_Greedy (one-step optimization) shows very unstable performance. For instance on Figures \ref{fig:harvard500PR}, \ref{fig:harvard500M}, \ref{fig:A80M}, \ref{fig:A800M} RLGL\_Greedy  has not performed well at all, whereas on Figure \ref{fig:twowheelsM} it converges extremely fast. Furthermore, clearly RLGL\_Greedy requires a much more effort in comparison with the other considered heuristics.
The above makes a strong case for searching for effective heuristics or for approximating the dynamic programming approach. Clearly, one-step optimization is not sufficient.
On the above mentioned examples, RLGL\_PC and RLGL\_Theta both continue to perform well.

In Figure~\ref{fig:3blocks-MF-opt} we apply various methods on the mean-field model of the SBM80 graph.
In particular, we apply RLGL with the optimal green-light scheduling by blocks as in Subsection \ref{ssec:DP} to attain $\varepsilon=10^{-14}$. Let us refer to this optimal scheduling as RLGL\_Opt. It is interesting to observe that, in order for RLGL\_Opt to reach the optimal turnpike, it needs to sacrifice significantly the convergence speed during the initial phase. As a matter of fact, during the initial phase the convergence rates of all the methods are higher than the  convergence rate of RLGL\_Opt. However once RLGL\_Opt reaches the optimal turnpike, its convergence rate becomes really impressive. We emphasize that the optimal green-light scheduling depends on the aimed precision.

After comparing the proposed heuristics, we recommend RLGL\_PC and RLGL\_Theta
as they both are easy to implement, scalable to very large graphs, typically outperform other state of the art methods, and can be distributed.
At the same time, our general observation is that myopic heuristics often have unstable performance. In fact,
we note that it is quite sufficient and preferable to give green light to nodes with a {\it large} amount of cash, not necessarily to the nodes with the {\it largest} amount of cash.

\section{Conclusions and future research}
\label{sec:conc}

We have proposed a versatile Red Light Green Light (RLGL) method for solution of large Markov chains.
Our inspiration had come from OPIC and Gauss-Southwell methods. However, there are crucial differences
between RLGL and the two above mentioned methods. Both OPIC and Gauss-Southwell are based on the online
distribution of positive cash. In the RLGL method we allow to distribute both {\em negative} as well as positive cash. This leads to at least exponential convergence rate as opposite to $O(1/t)$ rate in OPIC.
The Gauss-Southwell methods can be regarded as particular cases of our, more general, approach specified
for positive cash distribution in the context of PageRank. We have demonstrated that a
cash-dependent, green-light scheduling can result in significant computational gains.

In our opinion, this work opens a number of very interesting research directions. We are convinced that even more effective heuristics for green-light scheduling could be developed, than those proposed in this work. Elaboration of such heuristics  and their theoretical justification is an obvious future research direction. Next, we have observed that RLGL performs well on graphs
with clustered structure, that typically constitute a challenging setting for the standard methods
such as power iterations. It will be very interesting to further investigate  this phenomenon. Distributed implementations of the RLGL method represent another opportunity for further research that has high potential for practical applications. This topic clearly deserves a separate thorough study beyond our preliminary observations in  Section~\ref{sec:num}. Finally, we have observed that the optimisation of the green-light scheduling
can give an impressive computational gain but at the same time suffers from the curse of dimensionality.
One very interesting future research direction would be to look for approximate dynamic programming
approaches that could help to deal with this challenge.

\section*{Declarations}

\noindent {\bf Archiving:} This is the author version of the paper published
by Springer in {\it Journal of Scientific Computing}, v.93, 18 (2022).
{\tt https://doi.org/10.1007/s10915-022-01976-8}

\noindent {\bf Funding:} This work was partially supported by a grant from Qwant search engine company and
EU COST Action COSTNET CA15109. The work of NL is partially supported by the Netherlands Organisation for Scientific Research (NWO) through the Gravitation {\sc NETWORKS} grant no. 024.002.003.

\noindent {\bf Data availability:} The datasets analysed during the current study are available in the
following public repositories:
\begin{enumerate}
\item {\it SuiteSparse Matrix Collection}: {\tt http://sparse.tamu.edu/}
\item {\it SNAP Network Collection}: {\tt http://snap.stanford.edu/}
\item {\it WebGraph Laboratory Datasets}: {\tt http://law.di.unimi.it/},
see also Ref. \cite{WebGraph}.
\end{enumerate}

\bibliographystyle{plain}
\bibliography{RLGLarxiv}


%
%

\begin{figure}[h!]
\begin{center}
\begin{tikzpicture}
    \tikzstyle{every node}=[draw,circle,fill=black,minimum size=4pt,
                            inner sep=0pt]

    \draw (0,0) node (1) {}
        -- ++(0:2.0cm) node (2) {}
        -- ++(300:2.0cm) node (3) {}
        -- ++(240:2.0cm) node (4) {}
        -- ++(180:2.0cm) node (5) {}
        -- ++(120:2.0cm) node (6) {}
        -- (1); 

    \draw (1)
        -- ++(300:2.0cm) node (7) {};

    \draw (3)
        -- ++(0:2.0cm) node (8) {}
        -- ++(45:2.0cm) node (9) {}
        -- ++(-45:2.0cm) node (10) {}
        -- ++(-135:2.0cm) node (11) {}
        -- ++(90:1.41cm) node (12) {};

    \draw (2) -- (7);
    \draw (3) -- (7);
    \draw (4) -- (7);
    \draw (5) -- (7);
    \draw (6) -- (7);
    \draw (8) -- (12);
    \draw (9) -- (12);
    \draw (10) -- (12);
    \draw (8) -- (11);

\end{tikzpicture}
\end{center}
\caption{The two-wheels graph topology.}
\label{fig:twowheels}
\end{figure}
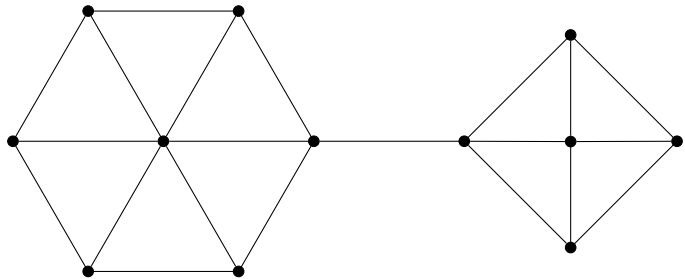

\begin{figure}[h!]
     \centering
     \begin{subfigure}[b]{0.49\textwidth}
       \centering
	   \includegraphics[width=0.99\textwidth]{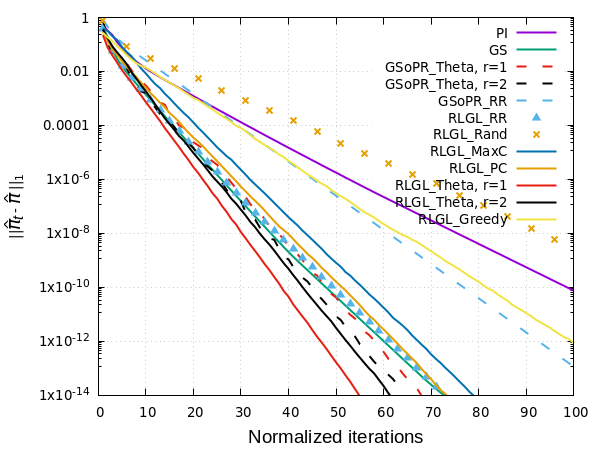}
       \caption{PageRank computation on the largest strongly connected component.}\label{fig:harvard500PR}
     \end{subfigure}
     \hfill
     \begin{subfigure}[b]{0.49\textwidth}
       \centering
	   \includegraphics[width=0.99\textwidth]{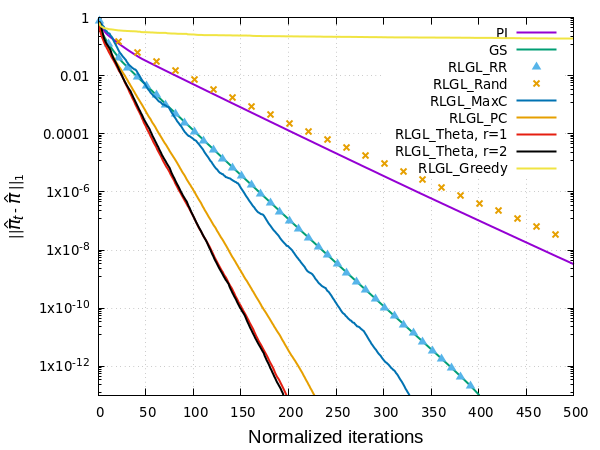}
       \caption{Stationary distribution computation on the largest strongly connected component.}\label{fig:harvard500M}
     \end{subfigure}
     \caption{Harvard500 graph.}\label{fig:harvard500}
\end{figure}


\begin{figure}[h!]
\centering
\includegraphics[width=7cm]{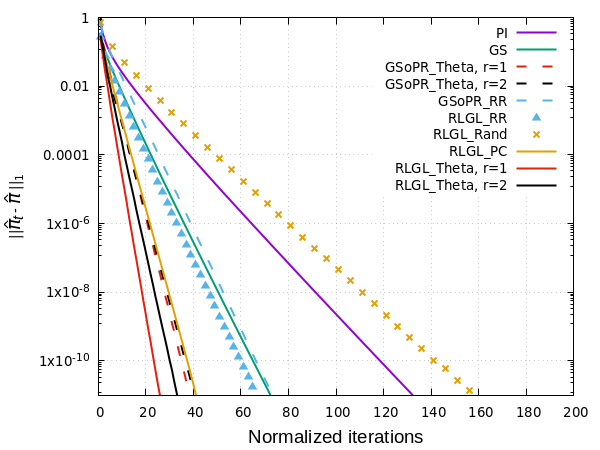}
\caption{PageRank computation on stanford graph.}
\centering
\label{fig:StanfordPR}
\end{figure}

\begin{figure}[h!]
     \centering
     \begin{subfigure}[b]{0.49\textwidth}
       \centering
	   \includegraphics[width=0.98\textwidth]{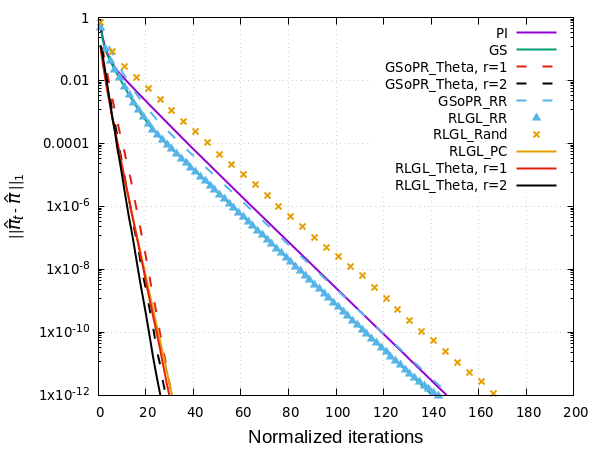}
       \caption{PageRank computation on uk-2007-05-100000.}\label{fig:uk-2007-05-100000PR}
     \end{subfigure}
     \hfill
     \begin{subfigure}[b]{0.49\textwidth}
       \centering
	   \includegraphics[width=0.98\textwidth]{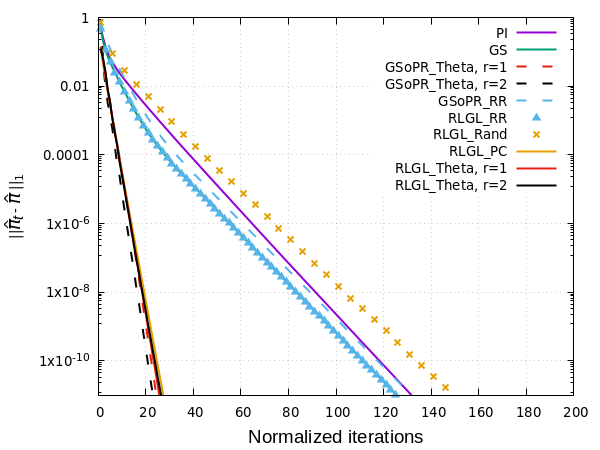}
       \caption{PageRank computation on uk-2007-05-1000000.}\label{fig:uk-2007-05-1000000PR}
     \end{subfigure}
     \caption{uk-2007 graphs.}\label{fig:uk-2007}
\end{figure}


\begin{figure}[h!]
\centering
\includegraphics[width=7cm]{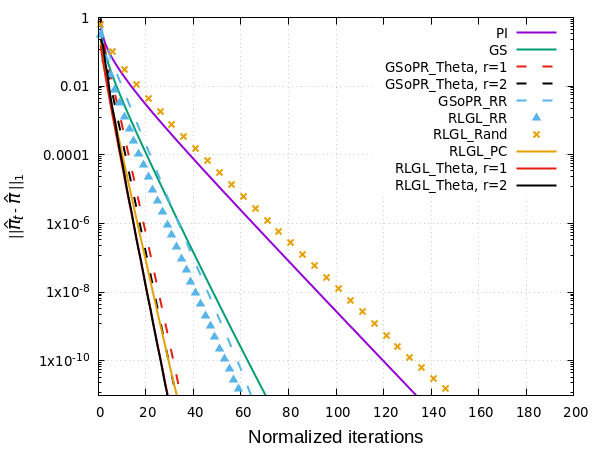}
\caption{PageRank computation on web-google graph.}
\centering
\label{fig:web-GooglePR}
\end{figure}

\begin{figure}[h!]
     \centering
     \begin{subfigure}[b]{0.49\textwidth}
       \centering
	   \includegraphics[width=0.98\textwidth]{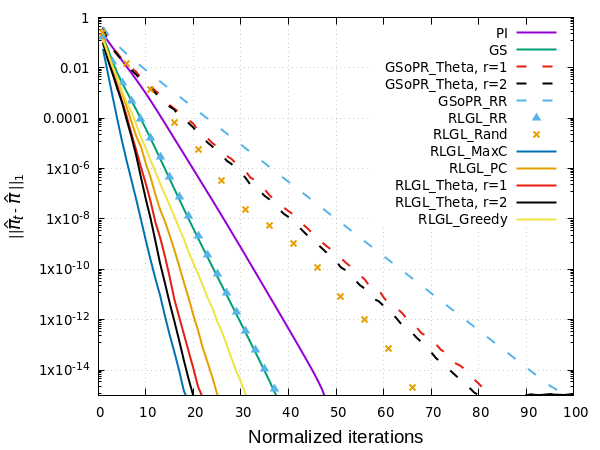}
       \caption{PageRank computation.}\label{fig:A80PR}
     \end{subfigure}
     \hfill
     \begin{subfigure}[b]{0.49\textwidth}
       \centering
	   \includegraphics[width=0.98\textwidth]{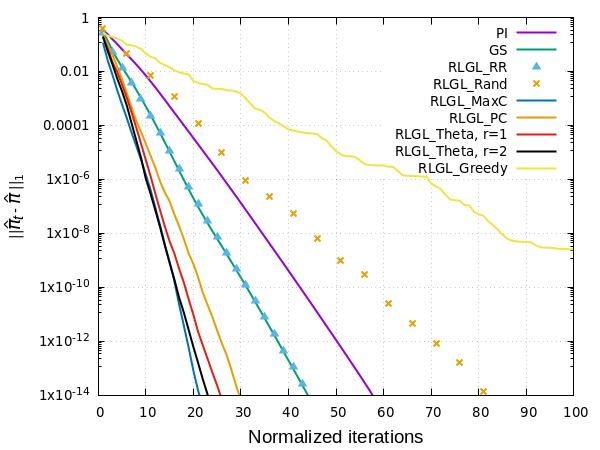}
       \caption{Stationary distribution computation.}\label{fig:A80M}
     \end{subfigure}
     \caption{SBM80 graph.}\label{fig:A80}
\end{figure}


\begin{figure}[h!]
     \centering
     \begin{subfigure}[b]{0.49\textwidth}
       \centering
	   \includegraphics[width=0.98\textwidth]{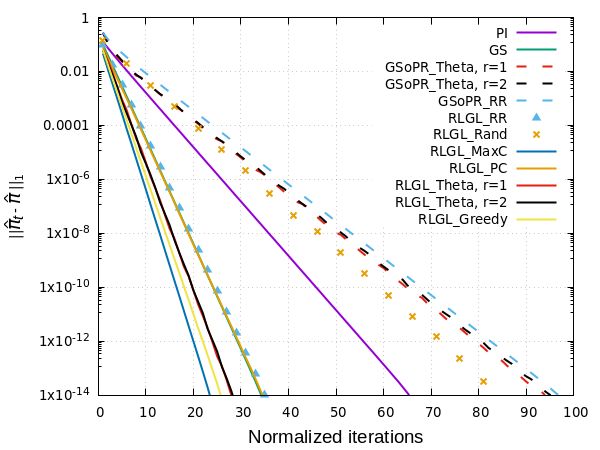}
       \caption{PageRank computation.}\label{fig:A800PR}
     \end{subfigure}
     \hfill
     \begin{subfigure}[b]{0.49\textwidth}
       \centering
	   \includegraphics[width=0.98\textwidth]{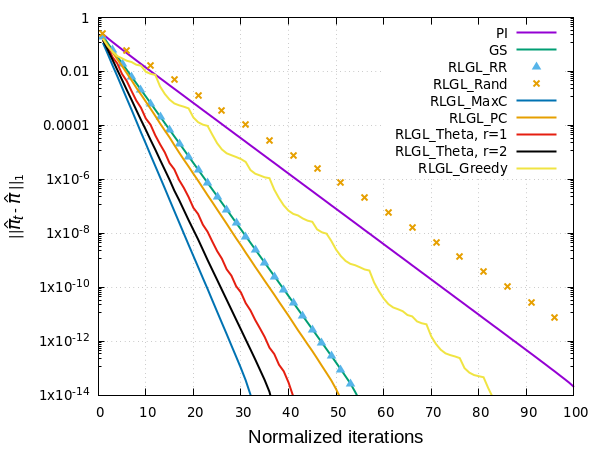}
       \caption{Stationary distribution computation.}\label{fig:A800M}
     \end{subfigure}
     \caption{SBM800 graph.}\label{fig:A800}
\end{figure}


\begin{figure}[h!]
     \centering
     \begin{subfigure}[b]{0.49\textwidth}
       \centering
	   \includegraphics[width=0.99\textwidth]{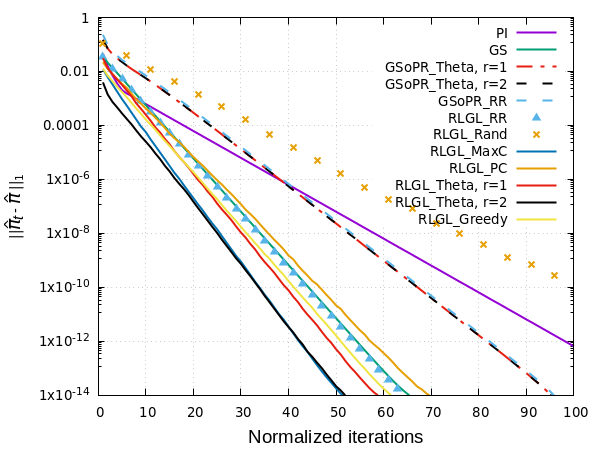}
       \caption{PageRank computation.}\label{fig:twowheelsPR}
     \end{subfigure}
     \hfill
     \begin{subfigure}[b]{0.49\textwidth}
       \centering
	   \includegraphics[width=0.99\textwidth]{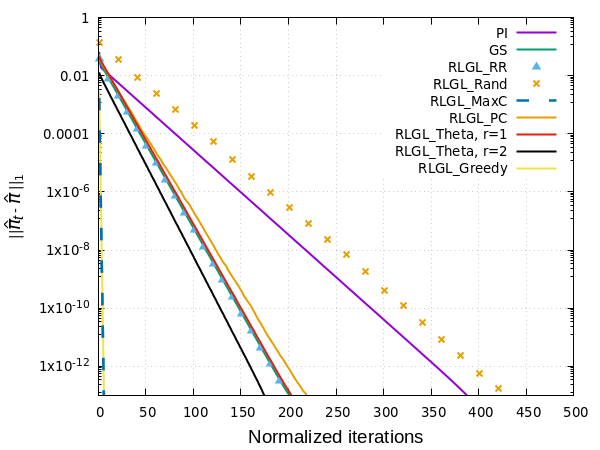}
       \caption{Stationary distribution computation.}\label{fig:twowheelsM}
     \end{subfigure}
     \caption{The two-wheels graph.}\label{fig:twowheelscomp}
\end{figure}


\begin{figure}[h!]
\centering
\includegraphics[width=8cm]{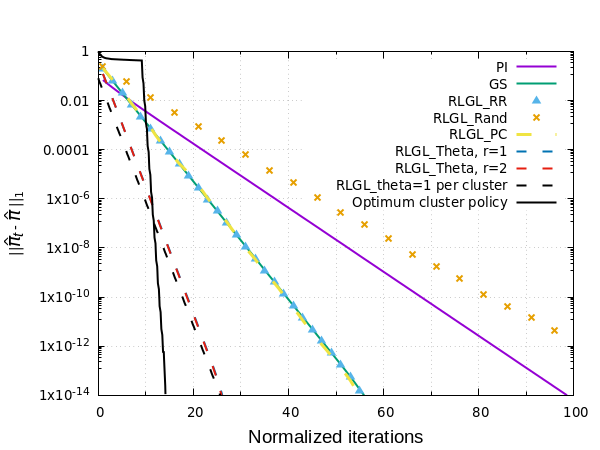}
\caption{Stationary distribution computation on the three block mean-field SBM.}
\centering
\label{fig:3blocks-MF-opt}
\end{figure}


\end{document}